\documentclass[12pt,reqno]{amsart}
\usepackage{exscale}
\usepackage{relsize}
\usepackage{amsfonts}
\usepackage{a4wide}
\usepackage{mathrsfs,amsmath,makecell}
\allowdisplaybreaks
\numberwithin{equation}{section}
\newtheorem{theorem}{Theorem}[section]
\newtheorem{proposition}[theorem]{Proposition}
\newtheorem{corollary}[theorem]{Corollary}
\newtheorem{lemma}[theorem]{Lemma}

\theoremstyle{definition}

\newtheorem{definition}[theorem]{Definition}

\newcommand{\R}{\mathbb{R}}

\begin{document}

\title
 [Normalized solutions to a class of Kirchhoff equations]
 {Normalized solutions to a class of Kirchhoff equations with Sobolev critical exponent}\footnotetext{*This work was supported by National Natural Science Foundation of China (Grant Nos. 11771166, 11901147) and Hubei Key Laboratory of Mathematical Sciences and Program for Changjiang Scholars and Innovative Research Team in University $\# $ IRT17R46.}

\maketitle
\begin{center}
\author{Gongbao Li}
\footnote{Corresponding Author: ligb@mail.ccnu.edu.cn (G. B. Li).}
\author{Xiao Luo}
\footnote{Email addresses: luoxiaohf@163.com (X. Luo).}
\author{Tao Yang}
\footnote{Email addresses: yangt@mails.ccnu.edu.cn (T. Yang).}
\end{center}

\begin{center}

\address {1, 3 School of Mathematics and Statistics, Central China Normal University, Wuhan, 430079, P. R. China}

\address {2 School of Mathematics, Hefei University of Technology, Hefei, 230009, P. R. China}

\end{center}
\maketitle

\begin{abstract}
In this paper, we consider the existence and asymptotic properties of solutions to the following Kirchhoff equation
\begin{equation}\label{1}\nonumber
- \Bigl(a+b\int_{{\R^3}} {{{\left| {\nabla u} \right|}^2}}\Bigl) \Delta u
   =\lambda u+ {| u |^{p - 2}}u+\mu {| u |^{q - 2}}u \text { in } \mathbb{R}^{3}
\end{equation}
under the normalized constraint $\int_{{\mathbb{R}^3}} {{u}^2}=c^2$,
where $a\!>\!0$, $b\!>\!0$, $c\!>\!0$, $2\!<\!q\!<\!\frac{14}{3}\!<\!
p\!\leq\!6$ or $\frac{14}{3}\!<\!q\!<\! p\!\leq\! 6$, $\mu\!>\!0$ and $\lambda\!\in\!\R$ appears as a Lagrange multiplier. In both cases for the range of $p$ and $q$, the Sobolev critical exponent $p\!=\!6$ is involved and the corresponding energy functional is unbounded from below on $S_c=\Big\{ u \in H^{1}({\mathbb{R}^3}): \int_{{\mathbb{R}^3}} {{u}^2}=c^2  \Big\}$. If $2\!<\!q\!<\!\frac{10}{3}$ and $\frac{14}{3}\!<\!
p\!<\!6$, we obtain a multiplicity result to the equation. If $2\!<\!q\!<\!\frac{10}{3}\!<\!
p\!=\!6$ or $\frac{14}{3}\!<\!q\!<\! p\!\leq\! 6$, we get a ground state solution to the equation. Furthermore, we derive several asymptotic results on the obtained normalized solutions.

Our results extend the results of N. Soave (J. Differential Equations 2020 $\&$ J. Funct. Anal. 2020), which studied the nonlinear Schr\"{o}dinger equations with combined nonlinearities, to the Kirchhoff equations. To deal with the special difficulties created by the nonlocal term
$({\int_{{\R^3}} {\left| {\nabla u} \right|} ^2}) \Delta u$ appearing in  Kirchhoff type equations, we develop a perturbed Pohozaev constraint approach and we find a way to get a clear picture of the profile of the fiber map via careful analysis. In the meantime, we need some subtle energy estimates under the $L^2$-constraint to recover compactness in the Sobolev critical case.





\end{abstract}

\maketitle

\section{Introduction and Main Result}

\setcounter{equation}{1}

This paper concerns the existence of solutions $(u,{\lambda})\in H^{1}({\mathbb{R}^3})\times {\mathbb{R}}$ to the following Kirchhoff equation
$$
- \Bigl(a+b\int_{{\R^3}} {{{\left| {\nabla u} \right|}^2}} \Bigr)\Delta u
   =\lambda u+ {| u |^{p - 2}}u+\mu {| u |^{q - 2}}u \text { in } \mathbb{R}^{3}
\eqno (1.1)_{\lambda}   $$
under the constraint
\begin{equation}\label{eq1.2}
  \int_{{\mathbb{R}^3}} {{u}^2}=c^2,
\end{equation}
where $a\!>\!0$, $b\!>\!0$, $c\!>\!0$, $2\!<\!q\!<\! p\!\leq\!6$ and $\mu\!>\!0$.

Letting $\lambda\in\R$, we say that a function $u\!\in\! H^{1}(\mathbb{R}^{3})$ is a weak solution to $(1.1)_{\lambda}$ if
$$\Bigl(a\!+\!b\int_{{\R^3}} {{{\left| {\nabla u} \right|}^2}} \Bigr)\int_{{\R^3}}\nabla u\nabla\varphi \!-\!\mu\int_{\mathbb{R}^{3}} \left|u\right|^{q\!-\!2} u {\varphi}\!-\! \int_{{\mathbb{R}^3}} \left|u\right|^{p\!-\!2} u {\varphi}\!-\!\lambda \int_{{\R^3}} u{\varphi}=0,~~~~\forall \varphi \in H^{1}(\mathbb{R}^{3}).$$
For fixed $\lambda$, equation $(1.1)_{\lambda}$ has been extensively studied, see e.g. \cite{dps,hl,hz2,ly,pz} and the references therein.

Alternatively, letting $c\!>\!0$ be fixed, we aim at finding a real number $\lambda \!\in\! {\mathbb{R}}$ and a function $u\!\in\! H^{1}({\mathbb{R}^3})$ solving $(1.1)_{\lambda}$ with ${||u||}_2\!=\!c$. Physicists call a solution $u$ of $(1.1)_{\lambda}$ with ${||u||}_2\!=\!c$ a normalized solution, and it can be obtained by searching critical points of the energy functional
\begin{align} \label{Enef}
E_{\mu}(u)= \frac{a}{2} {\| {\nabla u} \|}_2^2  + \frac{b}{4}{\| \nabla u \|}_2^4- \frac{1}{p}{\|u \|}_p^p-\frac{\mu}{q}{\| u \|}_q^q,~~~~\mu\geq 0
\end{align}
on the constraint
$$S_c: =\Big\{ u \in H^{1}({\mathbb{R}^3}): {||u||}_2^2=c^2  \Big\}$$
with Lagrange multipliers $\lambda$. We call $\frac{14}{3}$ the $L^2$-critical exponent for $(1.1)_{\lambda}$, since
$\inf _{u \in S_{c}} E_{\mu}(u)\!>\!-\infty$ if $q,p\in(2,\frac{14}{3})$ and $\inf _{u \in S_{c}} E_{\mu}(u)\!=\!-\infty$ if $\frac{14}{3}<q\leq6$ or $\frac{14}{3}<p\leq6$.

Taking $a\!=\!1$ and $b\!=\!0$, then $(1.1)_{\lambda}$ reduces to the classical Schr\"{o}dinger equation:
\begin{align}  \label{classicalNLSt}
-\Delta u=\lambda u+ {| u |^{p - 2}}u+\mu {| u |^{q - 2}}u \text { in } \mathbb{R}^{3}.
\end{align}
T. Cazenave and P.-L. Lions \cite{CaPl} and the very recent works of N. Soave \cite{NsAe,NSaE}, L. Jeanjean et al. \cite{eHtR}, L. Jeanjean and T. T. Le \cite{jeTl} are concerned with \eqref{classicalNLSt} in the more general cases
\begin{align} \label{classicalNLSG}
-\Delta u=\lambda u+ {| u |^{p - 2}}u+\mu {| u |^{q - 2}}u \text { in } \mathbb{R}^{N},
\end{align}
where $N\!\geq\!1$, $\mu\in\R$, $p\!\in\!(2,2^*]$, $q\!\in\!(2,2^*)$ and  $2^*:=\frac{2N}{(N-2)^{+}}$. It is worth pointing out that, L. Jeanjean and T. T. Le \cite{jeTl} solved an open question raised by N. Soave \cite{NSaE} if $N\geq4$. Some of their results on normalized solutions to (\ref{classicalNLSG}) are summarized in the following table: 
$$
\begin{array}   {|c|c|c|c|c|}
\hline \text { $N$ } & {\text { $\mu$ }} & {\text {$p$ and $q$}}
& \makecell[c]{\text{classifications of solutions}} &
{\text {references }} \\
\hline \text { $N\geq1$ } & {\text { $\mu> 0$ }} & {\text {$2<q<p\leq 2+\frac{4}{N}$}}  & {\text {a globla minimizer}} & {\text { \cite{CaPl,NsAe} }} \\
\hline \text { $N\geq1$ } & {\text { $\mu\!<\!0$ }} & {\text {$2<q\leq2+\frac{4}{N}<p<2^*$}}  & {\text {a Mountain Pass solution }} & {\text { \cite{NsAe} }} \\
\hline \text { $N\geq1$ } & {\text { $\mu>0$ }} & {\text {$2<q<2+\frac{4}{N}<p<2^*$}}  & \makecell[c]{\text{a local minimizer};\\ \text{a Mountain Pass solution}} & \text{\cite{NsAe}} \\
\hline \text { $N\geq3$ } & {\text { $\mu>0$ }} & {\text {$2<q<2+\frac{4}{N}$,~~$p=2^*$}} & {\text{a local minimizer }} &  {\textbf{ \cite{NSaE,eHtR} }}\\
\hline \text { $N\geq3$ } & {\text { $\mu>0$ }} & {\text {$2+\frac{4}{N}\leq q<2^*$,~~$p=2^*$}} & {\text{ a Mountain Pass solution }} &  {\textbf{ \cite{NSaE} }}\\
\hline \text { $N\geq4$ } & {\text { $\mu>0$ }} & {\text {$2<q<2+\frac{4}{N}$,~~$p=2^*$}}  & \makecell[c]{\text{a local minimizer};\\ \text{a Mountain Pass solution}} & \text{\cite{jeTl}}. \\
\hline
\end{array}
$$

Problem $(1.1)_{\lambda}$ also arises in the Kirchhoff type problem
\begin{align} \label{backgroud1.1}
- M\Bigl(\int_{\Omega} {{{\left| {\nabla u} \right|}^2}} \Bigr)\Delta u
   =f(x,u) \text { in } \Omega,~~~~~~~~u=0\text { on } \partial \Omega,
\end{align}
where $\Omega \!\subset\! \mathbb{R}^{3}$ is a smooth domain, $M\!:\!\R\!\to\!\R$ is some function and $f\!:\!\Omega \!\times\! \R\!\to\!\R$ is some nonlinearity. Recalling that \eqref{backgroud1.1} with $M(t) \!=\!a\!+\!bt$ ($a, b\!>\!0$)
is related to the stationary analogue of the equation
\begin{align} \label{backgroud1.3}
{u_{tt}} - (a + b\int_\Omega  {{{\left| {\nabla u} \right|}^2}} )\Delta u = f(x,u) \text { in } \Omega\times (0,+\infty),~~~~~~~~u(x,t)=0\text { on } \partial \Omega \times [0,+\infty).
\end{align}
In \cite{Kirc}, G. Kirchhoff introduced \eqref{backgroud1.3} as an extension of the D'Alembert wave equation
\[\rho \frac{{{\partial ^2}u}}{{\partial {t^2}}} - (\frac{{{\rho _0}}}{h} + \frac{E}{{2L}}\int_0^L {|\frac{{\partial u}}{{\partial x}}{|^2}dx} )\frac{{{\partial ^2}u}}{{\partial {x^2}}} = f(x,u)\]
for free vibrations of elastic strings, where $\rho$ denotes the mass density, $u$ the lateral displacement, $h$ the cross section area, $\rho _0$ the initial axial tension, $E$ the Young modulus, $L$ the length of the string and $f$ the external force. In particular, \eqref{backgroud1.1} with $M(0)\!=\!0$ models a string with zero initial tension, and is called the degenerate Kirchhoff equation, see \cite{UrMa,PRad}. One can refer to \cite{ap,ccs,dps,hl,hz2,ly,GmFn} and the references therein for more mathematical and physical background of \eqref{backgroud1.1}.

In \cite{y}, H. Y. Ye studied $(1.1)_{\lambda}\!-\!\eqref{eq1.2}$ with $a\!>\!0$, $b\!>\!0$, $\mu\!=\!0$ and $p \!\in\!(2,6)$. By considering a global minimization problem
$$m(c,0):=\inf _{u \in S_{c}} E_{0}(u)>-\infty, $$
she proved that $m(c,0)$ is attained if and only if $p \!\in\!(2,\frac{10}{3}]$ and $c\!>\!c^*$ or $p \!\in\!(\frac{10}{3},\frac{14}{3})$ and $c\!\geq\! c^*$, where
\[
c^*:=
\left\{ {\begin{array}{*{20}{c}}
{0,~~~~~~~~~~~~~~~~~~~~~~~~~~~~~~~~2<p<\frac{10}{3}};\\
{a^{\frac{3}{4}}\|W_p\|_2,~~~~~~~~~~~~~~~~~~~~~~~~~~~~~p=\frac{10}{3}};\\
{\inf\{{c \in (0,+\infty)}: m(c,0)<0\},~~~~~~~~\frac{10}{3}<p<\frac{14}{3}},
\end{array}}
\right.\]
(see Lemma \ref{lem2.2} below for $W_p$). When $p\!=\!\frac{14}{3}$, she showed that $m(c,0)$ has no minimizers for any $c\!>\!0$. Finally, she proved the existence of solutions to $(1.1)_{\lambda}\!-\!\eqref{eq1.2}$ by using the Pohozaev constraint method if $p\!\in\!(\frac{14}{3},6)$. Later on, H. L. Guo et al. in \cite{hLYw} proved that
$$c^*:=\Big[ 2 \|W_p\|_2^{p-2}  \Big(  \frac{2a}{14-3p} \Big)^{\frac{14-3p}{4}}    \Big(  \frac{b}{3p-10} \Big)^{\frac{3p-10}{4}}   \Big]^{\frac{1}{p(1-\delta_p)}}~~~~~~~~\mbox{if}~~~~~~~~\frac{10}{3}<p<\frac{14}{3}.$$
As subsequent works of \cite{y}, H. Y. Ye in \cite{Y2,Y3} considered the existence and mass concentration of critical points for $E_{0}|_{S_c}$ if $p \!=\!\frac{14}{3}$. She also studied $(1.1)_{\lambda}\!-\!\eqref{eq1.2}$ with an extra potential $V(x)$ in \cite{Ly1}. X. Y. Zeng et al. in \cite{XzYz} proved the existence and uniqueness of solutions to $(1.1)_{\lambda}\!-\!\eqref{eq1.2}$ with $a\!>\!0$, $b\!>\!0$, $\mu\!=\!0$ and $p \!\in\!(2,6)$ by using some simple energy estimates rather than the concentration-compactness principles adopted in \cite{y}.

To our best knowledge, the existence of normalized solutions to $(1.1)_{\lambda}$ with $a\!\geq\!0$, $b\!>\!0$, $\mu\!>\!0$, $p, q\!\in\!(2,6]$ and $p\!\not=\!q$ is still unknown. Without loss of generality, we set $q<p$ and consider problem $(1.1)_{\lambda}$ in the following two cases, respectively,  \\
\indent \indent$(i)$~~~~the mixed critical case:~~~~$a\!>\!0$, $b\!>\!0$, $c\!>\!0$, $\mu\!>\!0$ and $2\!<\!q\!<\!\frac{14}{3}\!<\! p\!\leq\!6;$ \\
\indent \indent$(ii)$ the purely $L^2$-supercritical case: $a\!>\!0$, $b\!>\!0$, $c\!>\!0$, $\mu\!>\!0$ and $\frac{14}{3}\!<\!q\!<\! p\!\leq\!6.$\\
It is worth pointing out that in both $(i)$ and $(ii)$, we cover the Sobolev critical case $p\!=\!6$. 

To state our main results, we introduce a definition and some frequently used constants. Following \cite{JpLn}, we say that $\tilde{u}\in{{H}}^{1}(\mathbb{R}^{3})$ is a ground state of $E_{\mu}|_{S_c}$ if
$$
d\left.E_{\mu}\right|_{S_{c}}(\tilde{u})=0 \quad \text { and } \quad E_{\mu}(\tilde{u})=\inf \left\{E_{\mu}(u): ~~~~d\left.E_{\mu}\right|_{S_{c}}(u)=0, ~~\text{and}~~u \in S_{c}\right\}.
$$
For $p,q\in(2,6]$, we introduce two frequently used constants:
\begin{align} \label{gs1.10}
\begin{gathered}
\delta_q\!=\!\frac{3(q\!-\!2)}{2q},~~~~ \delta_p\!=\!\frac{3(p\!-\!2)}{2p}.
\end{gathered}
\end{align}
Notice that $\delta_q,\delta_p\!\in\!(0,1)$ and $\delta_6\!=\!1$. In addition, we see that
$$4\!<\!q\delta_q\!<\!p\delta_p~~~~\mbox{if}~~~~
\frac{14}{3}\!<\!q\!<\! p\!<\! 6;~~~~~~~~q\delta_q\!<\!2\!<\!4\!<\!p\delta_p
~~~~\mbox{if}~~~~2\!<\!q\!<\!\frac{10}{3}~~~~\mbox{and}~~~~\frac{14}{3}\!<\!
p\!<\!6.$$

For $2\!<\!q\!<\!\frac{10}{3}$ and $\frac{14}{3}\!<\!p\!\leq\!6$, we denote:\\
\begin{align} \label{ConAt1.10}
&\nonumber \mu^*:=\Big[\frac{\frac{a}{2}\Big( \frac{bp}{4\mathcal{C}_p^{p}} \Big)^{  \frac{2-q\delta_q}{p\delta_p-4} }}{c^{ q(1-\delta_q)+\frac{p(1-\delta_p)(2-q\delta_q)}{p\delta_p-4} }} \!+\!\frac{{(\frac{b}{4})}^{  \frac{p\delta_p-q\delta_q}{p\delta_p-4} }  {(\frac{p}{\mathcal{C}_p^{p}})}^{  \frac{4-q\delta_q}{p\delta_p-4} } }{ c^{ q(1-\delta_q)+\frac{p(1-\delta_p)(4-q\delta_q)}{p\delta_p-4} } }\Big]\frac{q\mathcal{C}_{p,q}}{\mathcal{C}_q^{q}}; \\
& \nonumber \mu_{*}\!:=\!\Big[ \frac{q(p\delta_{p}-4)b} {4(p\delta_p-q\delta_{q})\mathcal{C}_q^{q} } \Big]{\Big[ \frac{p(4-q\delta_q)b}{4(p\delta_p-q\delta_q)\mathcal{C}_p^p } \Big]}^{\frac{4-q\delta_q}{p\delta_p-4}} \frac{1}{c^{ q(1-\delta_q)+\frac{p(1-\delta_p)(4-q\delta_q)}{p\delta_p-4} }};\\
&\mu^{**}\!:=\! \frac{2(\frac{b}{\delta_q})^{\frac{q\delta_q}{4}}}{ (6-q\delta_q)
\mathcal{C}_q^{q}}\cdot \left[ \frac{12q }{4-q\delta_q}\Big( \frac{a\mathcal{S}\Lambda}{3}\!+\!\frac{b\mathcal{S}^2{\Lambda}^2}{12} \Big) \right]^{1-\frac{q\delta_q}{4}}\frac{1}{c^{q(1-\delta_q)}},
\end{align}
where $\mathcal{C}_{p,q}\!:=\!\Big( \frac{8(4-q\delta_q)}{ p\delta_p(p\delta_p-2)(p\delta_p-q\delta_q) } \Big)^{  \frac{4-q\delta_q}{p\delta_p-4} }\!-\! \Big( \frac{8(4-q\delta_q)}{ p\delta_p(p\delta_p-2)(p\delta_p-q\delta_q) } \Big)^{  \frac{p\delta_p-q\delta_q}{p\delta_p-4} }\!>\!0$, $\Lambda\!=\!\frac{b{\mathcal{S}}^2}{2}
\!+\!\sqrt{a\mathcal{S}\!+\!\frac{b^2{\mathcal{S}}^4}{4}}$, the embedding constants $\mathcal{S}$ and $\mathcal{C}_p$ are given by
$$\mathcal{S}=\inf _{u \in {{D}}^{1,2}(\mathbb{R}^{3})\setminus \{0\} }    \frac{\left\|\nabla u\right\|_{2}^{2}}{||u||_{6}^{2}},~~~~~~~~\frac{1}{\mathcal{C}_p}=\inf _{u \in {{H}}^{1}(\mathbb{R}^{3})\setminus \{0\} }    \frac{\left\|\nabla u\right\|_{2}^{\delta_p} \left\|u\right\|_{2}^{(1-\delta_p)}}{  ||u||_{p} },$$
(see Section 2 below for details).
Let ${u}_{0}$ be the unique ground state of $E_{0}|_{S_c}$ (see Lemma \ref{lemma6.5}). In the mixed critical case $2\!<\!q\!<\!\frac{14}{3}\!<\!p\!\leq\!6$, our main results are the following Theorems \ref{th1.1}-\ref{th1.3}.

\begin{theorem}\label{th1.1}
Let $a\!>\!0$, $b\!>\!0$, $c\!>\!0$, $2\!<\!q\!<\!\frac{10}{3}$, $\frac{14}{3}\!<\!p\!<\!6$ and $0\!<\!\mu\!<\!\min\{\mu_{*},\mu^{*}\}$. Then \\
$\textbf{(1)}$ $E_{\mu}|_{S_c}$ has a critical point $\tilde{u}_{c,\mu}$ at some energy level $m(c, \mu) < 0$, which is a local minimizer of $E_{\mu}$ on the set
$$
A_{R_0} :=\left\{u \in S_c : {||\nabla u||}_2<R_0\right\}
$$
for a suitable $R_0=R_0(c,\mu)>0$. Moreover, $\tilde{u}_{c,\mu}$ is a ground state of $E_{\mu}|_{S_c}$, and any ground state of $E_{\mu}|_{S_c}$ is a local minimizer of $E_{\mu}$ on $A_{R_0}$; \\
$\textbf{(2)}$ $E_{\mu}|_{S_c}$ has a second critical point of Mountain Pass type $\hat{u}_{c,\mu}$ at some energy level $\sigma(c, \mu)>0$;\\
$\textbf{(3)}$ $\tilde{u}_{c,\mu}$ solves $(1.1)_{\tilde{\lambda}_{c,\mu}}$ and $\hat{u}_{c,\mu}$ solves $(1.1)_{\hat{\lambda}_{c,\mu}}$ for some $\tilde{\lambda}_{c,\mu},\hat{\lambda}_{c,\mu}\!<\!0$. Both $\tilde{u}_{c,\mu}$ and  $\hat{u}_{c,\mu}$ are positive and radially symmetric. Moreover, $\tilde{u}_{c,\mu}$ is radially deceasing;  \\
$\textbf{(4)}$ If $\tilde{u}_{c,\mu} \!\in \! S_{c}$ is a ground state for $E_{\mu}|_{S_{c}}$, then $m(c,\mu)\!\to\! 0^-$, $||\nabla \tilde{u}_{c,\mu}||_{2} \rightarrow 0$ as $\mu \rightarrow 0^{+}$;\\
$\textbf{(5)}$ $\sigma(c, \mu)\to {m}(c, 0)$ and $\hat{u}_{c,\mu} \to {u}_{0}$ in $H^1(\R^3)$ as $\mu \rightarrow 0^{+}$, where ${m}(c,0)=E_{0}({u}_{0})$ and ${u}_{0}$ is the unique ground state of $E_{0}|_{S_c}$.
\end{theorem}

\begin{theorem}\label{th1.3}
Let $a\!>\!0$, $b\!>\!0$, $c\!>\!0$, $2\!<\!q\!<\!\frac{10}{3}$, $p\!=\!6$ and $0\!<\!\mu\!<\!\min\{\mu_{*},\mu^{*},\mu^{**}\}$. Then\\
$\textbf{(1)}$ $E_{\mu}|_{S_c}$ has a critical point $\tilde{u}_{c,\mu}$ at some energy level $m(c, \mu) < 0$, which is a local minimizer of $E_{\mu}$ on the set
$$
A_{R_0} :=\left\{u \in S_{c} : {||\nabla u||}_2<R_0\right\}
$$
for a suitable $R_0=R_0(c,\mu)>0$. Moreover, $\tilde{u}_{c,\mu}$ is a ground state of $E_{\mu}|_{S_c}$, and any ground state of $E_{\mu}|_{S_c}$ is a local minimizer of $E_{\mu}$ on $A_{R_0}$; \\
$\textbf{(2)}$ $\tilde{u}_{c,\mu}$ solves $(1.1)_{\tilde{\lambda}_{c,\mu}}$ for some $\tilde{\lambda}_{c,\mu}\!<\!0$. Moreover, $\tilde{u}_{c,\mu}$ is positive and radially deceasing; \\
$\textbf{(3)}$ If $\tilde{u}_{c,\mu} \!\in \! S_{c}$ is a ground state for $E_{\mu}|_{S_{c}}$, then $m(c,\mu)\!\to\! 0^-$, $||\nabla \tilde{u}_{c,\mu}||_{2} \rightarrow 0$ as $\mu \rightarrow 0^{+}$.\\
\end{theorem}

In the purely $L^2$-supercritical case $\frac{14}{3}\!<\!q\!<\! p\!\leq\!6$, we have the following results.

\begin{theorem}\label{tTH1.3}
Let $a\!>\!0$, $b\!>\!0$, $c\!>\!0$, $\frac{14}{3}\!<\!q\!<\!p\!<\!6$ and $\mu\!>\!0$. Then \\
$\textbf{(1)}$ $E_{\mu}|_{S_c}$ has a critical point of Mountain Pass type $\hat{u}_{c,\mu}$ at a positive level $\sigma(c, \mu)\!>\!0$;\\
$\textbf{(2)}$ $\hat{u}_{c,\mu}$ is a positive radial solution to $(1.1)_{\hat{\lambda}_{c,\mu}}$ for suitable $\hat{\lambda}_{c,\mu}<0$. In addition, $\hat{u}_{c,\mu}$ is a ground state of $E_{\mu}|_{S_c}$; \\
$\textbf{(3)}$ $\sigma(c, \mu)\to {m}(c, 0)$ and $\hat{u}_{c,\mu} \to {u}_{0}$ in $H^1(\R^3)$ as $\mu \rightarrow 0^{+}$, where ${m}(c,0)=E_{0}({u}_{0})$ and ${u}_{0}$ is the unique ground state of $E_{0}|_{S_c}$.
\end{theorem}

\begin{theorem}\label{tTH1.4}
Let $a\!>\!0$, $b\!>\!0$, $c\!>\!0$, $\frac{14}{3}\!<\!q\!<\!6$, $p\!=\!6$ and $\mu\!>\!0$. Then \\
$\textbf{(1)}$ $E_{\mu}|_{S_c}$ has a critical point of Mountain Pass type $\hat{u}_{c,\mu}$ at level $\sigma(c, \mu)\!\in\!(0,\frac{a\mathcal{S}\Lambda}{3}
\!+\!\frac{b\mathcal{S}^2{\Lambda}^2}{12}
)$;\\
$\textbf{(2)}$ $\hat{u}_{c,\mu}$ is a positive radial solution to $(1.1)_{\hat{\lambda}_{c,\mu}}$ for suitable $\hat{\lambda}_{c,\mu}<0$. In addition, $\hat{u}_{c,\mu}$ is a ground state of $E_{\mu}|_{S_c}$; \\
$\textbf{(3)}$ $\sigma(c, \mu)\!\to\!\frac{a\mathcal{S}\Lambda}{3}\!+\!\frac{b\mathcal{S}^2{\Lambda}^2}{12}$,  $||{\hat{u}_{\mu}}||_6^2 \!\to\! {\Lambda}$, $||\nabla \hat{u}_{c,\mu}||^2_{2} \!\rightarrow\! \mathcal{S}\Lambda$ as $\mu \!\to\! 0^{+}$, where $\Lambda\!=\!\frac{b{\mathcal{S}}^2}{2}\!+\!
\sqrt{a\mathcal{S}\!+\!\frac{b^2{\mathcal{S}}^4}{4}}$.
\end{theorem}  


\noindent \textbf{Remark 1.1} Our results extend the results of N. Soave \cite{NsAe,NSaE}, which studied nonlinear Schr\"{o}dinger equations with combined nonlinearities, to the Kirchhoff equations. Compared with the cases $a\!+\!b\!>\!0$ and $ab\!=\!0$, our case $a\!>\!0$ and $b\!>\!0$ is more difficult since the corresponding fiber map $\Psi_{u}^{\mu}(s)$ has four different terms (see \eqref{eq1.13} below). In fact, it is delicate to precisely determine the numbers and types of critical points to $\Psi_{u}^{\mu}(s)$; in the meantime, the compactness analysis and energy estimates involving Sobolev critical exponent are very technical, since $b\!>\!0$ brings in the nonlocal term $({\int_{{\R^3}} {\left| {\nabla u} \right|} ^2}) \Delta u$. If $a\!=\!1$ and $b\!=\!0$, our results cover the existence results of \cite{NsAe,NSaE} in $3$-dimensional case; in particular, we see that $\frac{a\mathcal{S}\Lambda}{3}+\frac{b\mathcal{S}^2{\Lambda}^2}{12}
=\frac{\mathcal{S}^{\frac{3}{2}}}{3}$, which is nothing but the well-known critical energy threshold corresponding to $3$-dimensional Schr\"odinger equation. For the degenerate case $a\!=\!0$, the gap $\frac{10}{3}\!<\!q\!<\!\frac{14}{3}$ in Theorems \ref{th1.1}-\ref{th1.3} can be filled, since $\Psi_{u}^{\mu}(s)$ has only three different terms and its critical points are easily determined.


\noindent \textbf{Remark 1.2} If $2\!<\!q\!<\!\frac{10}{3}$ and $\frac{14}{3}\!<\!p\!<\!6$, we obtain two critical points for $E_{\mu}|_{S_c}$ in Theorem \ref{th1.1} because $E_{\mu}$ admits a convex-concave geometry provided $0\!<\!\mu\!<\!\mu^{*}$. The additional condition $\mu\!<\!\mu_{*}$ guarantees the Pohozaev manifold $\mathcal{P}_{c, \mu}$ is a natural constraint, on which the critical points of $E_{\mu}$ are indeed critical points for $E_{\mu}|_{S_c}$  (see Lemma \ref{lemma1.3} below). The condition $\mu\!<\!\mu^{**}$ in Theorem \ref{th1.3} is crucial in compactness analysis of the Palais-Smale sequences corresponding to $E_{\mu}|_{S_c}$. If $2\!<\!q\!<\!\frac{14}{3}$ and $p\!=\!6$, it is still a pending issue on how to obtain the second critical point for $E_{\mu}|_{S_c}$ even in the case $b=0$ (an open question raised by N. Soave \cite{NSaE}). For $b=0$, L. Jeanjean and T. T. Le \cite{jeTl} solved this open question if the dimension $N$ of the work space satisfies $N\geq4$. Therefore, the method of \cite{jeTl} is not applicable to our case since $N=3$. When it comes to the range $\frac{14}{3}\!<\!q\!<\!p\!\leq\!6$, the convex-concave geometry of $E_{\mu}$ disappears, we get at least one critical point for $E_{\mu}|_{S_c}$ in Theorems \ref{tTH1.3}-\ref{tTH1.4} because $E_{\mu}$ admits a Mountain Pass geometry.\\ 

The proofs of Theorems \ref{th1.1}-\ref{tTH1.4} are motivated by \cite{ScLj,NeJl,NsAe,NSaE}, which studied the Schr\"{o}dinger equations. In the $L^2$-supercritical regime, the global minimization method adopted in \cite{y} does not work and it is difficult to prove the boundedness of a Palais-Smale sequence corresponding to $E_{\mu}|_{S_c}$. Furthermore, the main obstacle for Kirchhoff-type problems is that we can not deduce
\begin{equation} \label{DIff}
 \lim\limits_{n\rightarrow \infty} ||\nabla u_n||^2_{2} \int_{\R^3}\nabla u_n \nabla \phi dx=||\nabla u||^2_{2}\int_{\R^3} \nabla u \nabla \phi dx,~~~~~~~~ \forall  \phi\in H^1(\R^3)
\end{equation}
only by $u_n\rightharpoonup u$ weakly in $H^1(\R^3)$. 

Usually, a bounded Palais-Smale sequence of $E_{\mu}|_{S_c}$ can be obtained by using the Pohozaev constraint approach (see \cite{ScLj,NeJl,NsAe,NSaE}). That is to say, we can construct a special Palais-Smale sequence $\{u_n\}\!\subset\! H_{rad}^{1}(\R^3)$ for $E_{\mu}|_{S_c}$ with
\begin{equation} \label{PohoKey}
P_{\mu}(u_n)=a||\nabla u_n||_{2}^2+b||\nabla u_n||_{2}^4-\mu\delta_{q}{||u_n||}_q^q-\delta_{p}{||u_n||}_p^p=o_n(1), 
\end{equation} 
then $\{u_n\}$ is bounded in $H^{1}(\R^3)$. Once proving $u_{n} \rightharpoonup u\not\equiv 0$ in $H^{1}(\R^3)$ for some $u \in H^{1}(\R^3)$, we can define
\begin{equation} \label{bgeqzero}
B\!:=\!\mathop {\lim }\limits_{n  \to \infty}||\nabla u_{n}||_{2}^{2}\!\geq\!||\nabla u||_{2}^{2}\!>\!0
\end{equation}
and hence \eqref{DIff} follows in a standard way if $p,q\!\in\!(2,6)$ (see Proposition \ref{prp4.1} below).

However, the Sobolev critical case $q\!\in\!(2,6)$ and $p\!=\!6$ is much different from the case $p,q\!\in\!(2,6)$. The proof of \eqref{bgeqzero} depends on solving a quartic polynomial equation. \textbf{We develop a perturbed Pohozaev constraint approach to prove \eqref{DIff}. Briefly speaking, the main observation is to rewrite $P_{\mu}\left(u_{n}\right)\!=\!o_n(1)$ (see \eqref{PohoKey}) as}
\begin{equation} \label{Key}
o_n(1)\!=\!P_{\mu}\left(u_{n}\right)\!=\! (a+B b)||\nabla u_{n}||_{2}^2\!-\!\mu\delta_{q}{||u||}_q^q\!-\! {||u_{n}||}_6^6\!+\!o_n(1),
\end{equation} 
where $B$ is defined in \eqref{bgeqzero}. The revision \eqref{Key} is the key point in proving \eqref{DIff}, since it possesses the splitting properties of the Br\'{e}zis-Lieb lemma (see \cite{a7}). Then, a subtle compactness analysis of $\{u_n\}$ leads to \eqref{DIff} (see Proposition \ref{prp4.2} below).

It remains to search a suitable Palais-Smale sequence $\{u_n\}\!\subset\! H_{rad}^{1}(\R^3)$ for $E_{\mu}|_{S_c}$. 
To this end, we need to know a clear picture of the corresponding fiber map $\Psi_{u}^{\mu}(s)$ (see \eqref{eq1.13} below). This process is quite different from that adopted in \cite{NsAe,NSaE} since the appearance of the nonlocal term $({\int_{{\R^3}} {\left| {\nabla u} \right|} ^2}) \Delta u$. \textbf{We reach this goal by a careful analysis of the profile of some polynomials (see Lemma \ref{LemA2.7} and Lemma \ref{LeMal2.8}).} 


The rest is standard as in \cite{NsAe,NSaE}. In the case of $2\!<\!q\!<\!\frac{10}{3}$ and $\frac{14}{3}\!<\!p\!\leq\!6$, we first study a local minimization problem $m(c,\mu)\!:=\!\inf _{u \in A_{R_{0}}} E_{\mu}(u)$ for some $R_0\!>\!0$. By using rearrangement technique and the Ekeland's variational principle, we get a desired Palais-Smale sequence $\{u_n\}$ for $E_{\mu}|_{S_{c}}$ at energy level $m(c,\mu)<0$. The compactness of $\{u_n\}$ guarantees the existence of a local minimizer for $E_{\mu}|_{A_{R_{0}}}$ if $2\!<\!q\!<\!\frac{10}{3}$ and $\frac{14}{3}\!<\!p\!<\!6$. Utilizing $m(c,\mu)$ and a min-max principle (see Lemma 2.7), we also get a Mountain Pass type critical point for $E_{\mu}|_{S_c}$. If $2\!<\!q\!<\!\frac{10}{3}$ and $p\!=\!6$, we recover the compactness of $\{u_n\}$ by using $\mu\!<\!\mu^{**}$ and $m(c,\mu)\!<\!0$.

In the case of $\frac{14}{3}\!<\!q\!<\!p\!\leq\!6$, we obtain a Mountain Pass critical point for $E_{\mu}|_{S_c}$ at energy level $\sigma(c, \mu)$ by a min-max principle. The selected Palais-Smale sequence $\{u_n\}$ for $E_{\mu}|_{S_c}$ is compact provided $\frac{14}{3}\!<\!q\!<\!p\!<\!6$. However, we need the extra energy estimate $\sigma(c, \mu)\!<\!\frac{a\mathcal{S}\Lambda}{3}\!
+\!\frac{b\mathcal{S}^2{\Lambda}^2}{12}$ to recover the compactness of $\{u_n\}$ when $\frac{14}{3}\!<\!q\!<\!6$ and $p\!=\!6$. \textbf{Since $b>0$ and the min-max procedure is confined by the $L^2$-constraint, the proof of $\sigma(c, \mu)\!<\!\frac{a\mathcal{S}\Lambda}{3}\!
+\!\frac{b\mathcal{S}^2{\Lambda}^2}{12}$ is very delicate} (see Lemma \ref{lemma7.5} below).\\

This paper is organized as follows, in Section 2, we give some preliminaries. In Section 3, we give the compactness analysis of Palais-Smale sequences for $\left.E_{\mu}\right|_{S_{c}}$. In Section 4, we consider the mixed critical case and prove Theorems \ref{th1.1}-\ref{th1.3}. In Section 5, we study the purely $L^2$-supercritical case and prove Theorems \ref{tTH1.3}-\ref{tTH1.4}.
 \\

\textbf{Notations:}~~~~Throughout this paper, we use standard notations. The integral $\int_{{\R^3}} fdx$ is simply denoted by $\int_{{\R^3}} f$. For $1 \!\le\! p \!<\!\infty $ and $u\!\in\!{L^p}({\R^3})$, we denote ${\left\| u \right\|_p}\!:=\! {({\int_{{\R^3}} {\left| u \right|} ^p})^{\frac{1}{p}}}$. The Hilbert space $H^{1}(\mathbb{R}^{3})$ is defined as
$$H^{1}(\mathbb{R}^{3}) :=\{u \in L^{2}(\mathbb{R}^{3}) :\nabla u \in L^{2}(\mathbb{R}^{3})\}$$
with the inner product $(u,v): = \int_{{\R^3}} {\nabla u\nabla v}  + \int_{{\R^3}} {uv}$  and norm ${\left\| u \right\|} := (\left\| {\nabla u} \right\|_2^2 + \left\| u \right\|_2^2)^{\frac{1}{2}}$.
$H^{-1}({\R^3})$ is the dual space of $H^1({\R^3})$. The space $D^{1,2}(\mathbb{R}^{3})$ is defined as
$$D^{1,2}(\mathbb{R}^{3}) :=\{u \in L^6(\mathbb{R}^{3}) :\nabla u \in L^{2}(\mathbb{R}^{3})\},$$
which is in fact the completion of $C_{0}^{\infty}(\R^3)$ under the norm
$   ||u||_{D^{1,2}(\mathbb{R}^{3})}\!=\!\left\| {\nabla u} \right\|_2$.
For $N\!\geq\!1$, $H_{rad}^{1}(\mathbb{R}^{N}) \!:=\!\{u(x) \!\in\! H^{1}(\mathbb{R}^{N}): u(x)\!=\!u(|x|)\}$, $H_{+}^{1}(\R^N)\!:=\!\{u(x) \!\in\! H^{1}(\mathbb{R}^{N}): u(x)\!\geq\!0\}$ and $S_{c,r}: =H_{rad}^1 \cap S_c=\Big\{ u \in H_{rad}^1({\mathbb{R}^3}): {||u||}_2^2=c^2  \Big\}$. We use $``\rightarrow"$ and $``\rightharpoonup"$ to denote the strong and weak convergence in the related function spaces respectively. $C$ and $C_{i}$ will denote positive constants. $\langle\cdot,\cdot\rangle$ denote the dual pair for any Banach space and its dual space.
 $X \hookrightarrow Y$ means $X$ embeds into $Y$. $o_{n}(1)$ and $O_{n}(1)$ mean that $|o_{n}(1)|\to 0$ and $|O_{n}(1)|\leq C$ as $n\to+\infty$, respectively.


\section{Preliminaries}

\setcounter{equation}{0}
In this Section, we give some preliminaries. The next lemma is the Sobolev embedding.
\begin{lemma}(\cite{GtEi}) \label{lem2.1}
There exists a constant $\mathcal{S}>0$ such that
\begin{equation} \label{equ2.1}
\mathcal{S}=\inf _{u \in {{D}}^{1,2}(\mathbb{R}^{3})\setminus \{0\} }    \frac{\left\|\nabla u\right\|_{2}^{2}}{||u||_{6}^{2}}.
\end{equation}
\end{lemma}

\begin{lemma} (Gagliardo-Nirenberg inequality, \cite{Wein}) \label{lem2.2}
Let $p\!\in\!(2,6)$. Then there exists a constant $\mathcal{C}_{p}\!=\!\Big( \frac{p}{2 ||W_p||^{p-2}_{2}} \Big)^{\frac{1}{p}}\!>\!0$ such that
\begin{equation} \label{equ2.2}
 ||u||_{p} \leq \mathcal{C}_{p} \left\|\nabla u\right\|_{2}^{\delta_p} \left\|u\right\|_{2}^{(1-\delta_p)}, \qquad \forall u \in {H}^{1}(\mathbb{R}^{3})
\end{equation}
where $\delta_p\!=\!\frac{3(p-2)}{2p}$ and $W_p$ is the unique positive solution of
$ -\Delta W\!+\!(\frac{1}{\delta_p}-\!1)W \!=\!\frac{2}{p\delta_p}|W|^{p-2}W$.
\end{lemma}

For any $ u \in S_c$, (\ref{equ2.2}) 
indicates that $\inf _{u \in S_c} E_{\mu}(u)>-\infty$ if $p,q\in(2,\frac{14}{3})$. On the contrary, we have $\inf _{u \in S_c} E_{\mu}(u)=-\infty$ for $\frac{14}{3}<q\leq6$ or $\frac{14}{3}< p \leq 6$, and therefore the global minimization method used in \cite{y} does not work any more. Naturally, we would hope to overcome this difficulty by using the Pohozaev constraint method adopted in \cite{NsAe,NSaE}. To this end, we need the following lemma which is related to the Pohozaev identity.
\begin{lemma}  \label{lem2.7}
Let $a\!\geq\!0$, $b\!>\!0$, $p,q \!\in\!(2,6]$ and $\mu, \lambda\!\in\!\R$. If $u \!\in\! {H}^{1}(\mathbb{R}^{3})$ is a weak solution of
\begin{align}  \label{al2.7}
 - \Bigl(a+b\int_{{\R^3}} {{{\left| {\nabla u} \right|}^2}} \Bigr)\Delta u
   =\lambda u+ {| u |^{p - 2}}u+\mu {| u |^{q - 2}}u \text { in } \mathbb{R}^{3},
\end{align}
then the Pohozaev identity
$P_{\mu}(u)\!:=\!a||\nabla u||_{2}^2\!+\!b||\nabla u||_{2}^4\!-\!\mu\delta_{q}{||u||}_q^q\!-\!\delta_{p} {||u||}_p^p\!=\!0$  holds.
\end{lemma}
\begin{proof}
If $u\!\equiv\!0$, then $P_{\mu}(u)\!=\!0$. If $u\!\not\equiv\!0$, \eqref{al2.7} becomes $-(a+bB) \Delta u\!=\!\lambda u\!+\! {| u |^{p - 2}}u\!+\!\mu {| u |^{q - 2}}u$
for $B\!=\!\int_{{\R^3}} {{{\left| {\nabla u} \right|}^2}}$, then the elliptic regularity theory implies that $u \in C^2(\R^3)$. The rest is standard as in \cite{PusE}.
\end{proof}
When $\inf _{u \in S_{c}} E_{\mu}(u)=-\infty$, we introduce the Pohozaev set:
\begin{equation}  \label{eq2.4}
\mathcal{P}_{c, \mu}=\left\{u \in S_{c} : 0\!=\!P_{\mu}(u)\!=\!a||\nabla u||_{2}^2\!+\!b||\nabla u||_{2}^4\!-\!\mu\delta_{q}{||u||}_q^q\!-\!\delta_{p} {||u||}_p^p \right\}.
\end{equation}
Lemma \ref{lem2.7} implies that any critical point of $E_{\mu}|_{S_c}$ is contained in $\mathcal{P}_{c, \mu}$. For $u \!\in \!S_c$ and $s\!\in\! \mathbb{R}$, we define
\begin{equation} \label{eq1.12}
(s \star u)(x) :=e^{\frac{3}{2} s} u\left(e^{s} x\right).
\end{equation}
Then, $s \star u \in S_c$ and that the map $(s, u) \in \mathbb{R} \times H^{1}(\R^3) \mapsto s \star u \in H^{1}(\R^3)$ is continuous (see Lemma 3.5 in \cite{bTEv}). Let $u \!\in \!S_c$ and $\mu\!\in\! \mathbb{R}^+$ be fixed, we define the fiber map
\begin{equation} \label{eq1.13}
\Psi_{u}^{\mu}(s) :=E_{\mu}(s \star u)=\frac{a}{2}e^{2s} ||\nabla u||_{2}^2+\frac{b}{4}e^{4s} ||\nabla u||_{2}^4-\mu \frac{e^{ q\delta_{q} s}}{q} {||u||}_q^q-\frac{e^{ p\delta_{p}s}}{p} {||u||}_p^p,~~~~~~~~\forall s\!\in\! \mathbb{R}.
\end{equation}
Direct calculation gives
\begin{equation} \label{equa2.9}
\left(\Psi_{u}^{\mu}\right)^{\prime}(s)=ae^{2s} ||\nabla u||_{2}^2+be^{4s} ||\nabla u||_{2}^4-\mu \delta_{q} e^{ q\delta_{q} s} {||u||}_q^q-\delta_{p}e^{ p\delta_{p}s} {||u||}_p^p=P_{\mu}(s \star u).
\end{equation}
Therefore, $\left(\Psi_{u}^{\mu}\right)^{\prime}(s)=0$ if and only if $s \star u \in \mathcal{P}_{c, \mu}$. From (\ref{equa2.9}), we see immediately that:

\begin{corollary} \label{coroll2.1}
Let $u \!\in \!S_c$ and $\mu\!\in\! \mathbb{R}^+$. Then $s \!\in\! \mathbb{R}$ is a critical point for $\Psi_{u}^{\mu}$ if and only if $s \!\star \!u \!\in\! \mathcal{P}_{c, \mu}$.\\
\end{corollary}

To determine the exact location and types of some critical points for $E_{\mu}|_{S_c}$, we observe that $\mathcal{P}_{c, \mu}$ can be split into the disjoint union $\mathcal{P}_{c, \mu}=\mathcal{P}_{+}^{c, \mu}\cup \mathcal{P}_{0}^{c, \mu} \cup \mathcal{P}_{-}^{c, \mu} $, where
$$\mathcal{P}_{+}^{c, \mu}\! :=\!\left\{u \in \mathcal{P}_{c, \mu} :\left(\Psi_{u}^{\mu}\right)^{\prime \prime}(0)\!>\!0\right\}, ~~~~\mathcal{P}_{-}^{c, \mu} \!:=\!\left\{u \in \mathcal{P}_{c, \mu} :\left(\Psi_{u}^{\mu}\right)^{\prime \prime}(0)\!<\!0\right\},$$
$\mathcal{P}_{0}^{c, \mu} \!:=\!\left\{u \in \mathcal{P}_{c, \mu} :\left(\Psi_{u}^{\mu}\right)^{\prime \prime}(0)\!=\!0\right\}$ for $\left(\Psi_{u}^{\mu}\right)^{\prime \prime}(0)\!:= \!2a||\nabla u||_{2}^2\!+\!4b||\nabla u||_{2}^4\!-\!\mu q \delta_{q}^2{||u||}_q^q\!-\!p\delta_{p}^2 {||u||}_p^p$. \\

We also need the following lemma.
\begin{lemma} (\cite{bTEv}, Lemma 3.6) \label{lem2.18}
For $u \!\in\! S_c$ and $s\!\in\! \mathbb{R}$, the map $\varphi  \!\mapsto\! s \star \varphi$ from $T_{u}S_c$ to $T_{s \star u}S_c$
is a linear isomorphism with inverse $\psi\!\mapsto\!(-s)\! \star \!\psi$, where $T_uS_c\!:=\!\{\varphi \!\in\! S_c: \int_{{\R^3}} u{\varphi}\!=\!0 \}$.
\end{lemma}


\begin{definition}\label{def3.9}
Let $X$ be a topological space and $B$ be a closed subset of $X$. We shall say that a class $\mathcal{F}$ of compact subsets of $X$ is a homotopy-stable family with extended boundary $B$ if for any set $A$ in $\mathcal{F}$ and any $\eta\in C([0,1]\times X;X)$ satisfying $\eta(t,x)=x$ for all $(t,x)\in (\{0\}\times X)\cup ([0,1]\times B)$ we have that $\eta(\{1\}\times A)\in \mathcal{F}$.
\end{definition}
The following Lemma \ref{lem3.10} is a min-max principle obtained by N. Ghoussoub \cite{GN}.

\begin{lemma}(\cite{GN}, Theorem 5.2)\label{lem3.10}
Let $\varphi$ be a $C^{1}$-functional on a complete connected $C^{1}$-Finsler manifold $X$ and consider a homotopy-stable family $\mathcal{F}$ with an extended closed boundary $B$. Set $m=m(\varphi,\mathcal{F})$ and let $F$ be a closed subset of $X$ satisfying \\
\indent $(1)$~~~~~~~~$(A \cap F)\backslash B \neq \emptyset \quad \text { for every } A \in \mathcal{F}$, \\
\indent $(2)$~~~~~~~~$\sup \varphi(B) \leq m \leq \inf \varphi(F)$.  \\
Then, for any sequence of sets $(A_{n})_{n}$ in $\mathcal{F}$ such that $\lim_{n}\sup_{A_{n}}\varphi=m$, there exists a sequence $(x_{n})_{n}$ in $X$ such that
$$\lim_{n \rightarrow +\infty}\varphi(x_{n})=m,\ \ \lim_{n \rightarrow +\infty}\|d\varphi(x_{n})\|=0,\ \ \lim_{n \rightarrow +\infty}dist(x_{n},F)=0,\ \ \lim_{n \rightarrow +\infty}dist(x_{n},A_{n})=0.$$
\end{lemma}

\section{Compactness analysis of Palais-Smale sequences for $\left.E_{\mu}\right|_{S_{c}}$}
In this Section, we give the compactness analysis of Palais-Smale sequences for $\left.E_{\mu}\right|_{S_{c}}$. The next two propositions are motivated by \cite{NsAe,NSaE}, which studied nonlinear Schr\"{o}dinger equations ($a\!=\!1$, $b\!=\!0$ in our cases). To deal with the special difficulties created by the nonlocal term $({\int_{{\R^3}} {\left| {\nabla u} \right|} ^2}) \Delta u$, we develop a perturbed Pohozaev constraint approach in proving  Proposition \ref{prp4.2}.

In the Sobolev subcritical case $p,q\!\in\!(2,6)$, we have
\begin{proposition} \label{prp4.1}
Let $a\!>\!0$, $b\!>\!0$, $c\!>\!0$, $\mu\!>\!0$, $2\!<\!q\!<\!\frac{14}{3}\!<\!p\!<\!6$ or $\frac{14}{3}\!<\!q\!<\!p\!<\!6$. Let $\left\{u_{n}\right\} \subset S_{c, r}$ be a Palais-Smale
sequence for $\left.E_{\mu}\right|_{S_{c}}$ at energy level $m \not=0$ with $P_{\mu}\left(u_{n}\right) \rightarrow 0$ as $n \rightarrow \infty$. Then up to a subsequence $u_{n} \rightarrow u$ strongly in $H^{1}(\R^3)$ for some $u \in H^{1}(\R^3)$. Moreover, $u\in S_{c}$ and $u$ is a radial solution to $(1.1)_{\lambda}$ for some $\lambda<0$.
\end{proposition}
\begin{proof}
The proof is divided into four main steps.

\noindent \textbf{(1) Boundedness of $\{u_n\}$ in $H^{1}(\R^3)$}. If $2\!<\!q\!<\!\frac{14}{3}\!<\!p\!<\!6$, we have $q\delta_q\!<\!4\!<\!p\delta_p$ and  $$
E_{\mu}\left(u_{n}\right)\!=\!(\frac{a}{2}-\frac{a}{p\delta_p})||\nabla u_{n}||_{2}^{2}\!+\!(\frac{b}{4}-\frac{b}{p\delta_p})||\nabla u_{n}||_{2}^{4}\!-\!\frac{\mu}{q}
\left(1-\frac{q\delta_q}{p\delta_p}\right)\|u_n\|_{q}^q\!+\!o_n(1)$$
by $P_{\mu}(u_n)\!=\!o_n(1)$. 
It results to
\begin{align*}
(\frac{a}{2}-\frac{a}{p\delta_p})||\nabla u_{n}||_{2}^{2}+(\frac{b}{4}-\frac{b}{p\delta_p})||\nabla u_{n}||_{2}^{4}
 \leq (m+1)&+\frac{\mu}{q}
\left(1-\frac{q\delta_q}{p\delta_p}\right)\mathcal{C}_q^{q} \left\|\nabla u_n\right\|_{2}^{q\delta_q} c^{q(1-\delta_q)},
\end{align*}
which gives $||\nabla u_{n}||_{2}\leq C$. If $\frac{14}{3}\!<\!q\!<\!p\!<\!6$, we have $4\!<\!q\delta_q\!<\!p\delta_p$ and  $E_{\mu}\left(u_{n}\right)\!=\!\frac{a}{4}||\nabla u_{n}||_{2}^{2}\!+\!(\frac{\delta_p}{4}-\frac{1}{p})\|u_n\|_{p}^p\!
+\!\mu(\frac{\delta_q}{4}-\frac{1}{q})\|u_n\|_{q}^q\!+\!o_n(1)\leq (m+1)$. So $\{u_n\}$ is bounded in $H^{1}(\R^3)$.

\noindent \textbf{(2) $\exists$ Lagrange multipliers $\lambda_{n} \rightarrow \lambda \in \mathbb{R}$.} Since $H_{\mathrm{rad}}^{1}\left(\mathbb{R}^{3}\right) \hookrightarrow L^{r}\left(\mathbb{R}^{3}\right)$ is compact for $r \in\left(2,6\right)$, we deduce that there exists an $u \in H_{\mathrm{rad}}^{1}$ such that, up to a subsequence,
$$u_{n} \rightharpoonup u~~~~\mbox{in}~~~~H^{1}(\mathbb{R}^{3}),~~~~u_{n} \rightarrow u~~~~\mbox{in}~~~~L^{r}(\mathbb{R}^{3}),~~~~u_{n} \rightarrow u~~~~\mbox{a.e. on}~~~~\mathbb{R}^{3}.$$
Notice that $\left\{u_{n}\right\}$ is a Palais-Smale sequence of $\left.E_{\mu}\right|_{S_{c}}$, by the Lagrange multipliers rule there exists $\lambda_{n} \in \mathbb{R}$ such that
\begin{equation} \label{eq4.2}
\Bigl(a+b{\| {\nabla u_n} \|}_2^2 \Bigr)\int_{{\R^3}}\nabla u_n \nabla\varphi \!-\!\mu\int_{\mathbb{R}^{3}} \left| u_n \right|^{q\!-\!2} u_n {\varphi}\!-\! \int_{{\mathbb{R}^3}} \left|u_n\right|^{p\!-\!2} u_n {\varphi}\!-\!\lambda_n \int_{{\R^3}} u_n{\varphi}=o_n(1)
\end{equation}
for every $\varphi \in H^{1}(\R^3)$, where $o_n(1) \rightarrow 0$ as $n \rightarrow \infty$. In particular, take $\varphi=u_n$, then
$$\lambda_{n} c^{2}=a{\| {\nabla u_n} \|}_2^2+b{\| {\nabla u_n} \|}_2^4-
\mu{||u_n||}_q^{q}-{||u_n||}_p^{p}+o_n(1).$$
The boundedness of $\left\{u_{n}\right\}$ in $H^{1} \cap L^{q} \cap L^{p} $ implies that $\lambda_{n} \rightarrow \lambda \in \mathbb{R}$, up to a subsequence.

\noindent \textbf{(3) $\lambda<0$ and $u\not \equiv0$.} Recalling that $P_{\mu}\left(u_{n}\right) \rightarrow 0$, we have
$$\lambda_{n} c^{2}=\mu(\delta_q-1)||u_{n}||_{q}^{q}+(\delta_p-1) ||u_{n}||_{p}^{p}+o_n(1).$$
Letting $n\to+\infty$, then $\lambda c^{2}=\mu(\delta_q-1)||u||_{q}^{q}+(\delta_p-1) ||u||_{p}^{p}$. Since $\mu>0$ and $0<\delta_{q},\delta_p<1$, we deduce that $\lambda \leq 0$, with ``=" if and only if $u\equiv0$. If $\lambda_{n} \rightarrow 0$, we have $\mathop {\lim }\limits_{n  \to \infty}||u_{n}||_{p}^{p}=0=\mathop {\lim }\limits_{n  \to \infty}||u_{n}||_{q}^{q}$. Using again $P_{\mu}\left(u_{n}\right) \rightarrow 0$, we have $E_{\mu}\left(u_{n}\right)\to0$.
A contradiction with $E_{\mu}\left(u_{n}\right) \rightarrow m \neq 0$ and
thus $\lambda_{n} \rightarrow \lambda<0$ and $u\not \equiv0$.

\noindent \textbf{(4) $u_{n} \rightarrow u$ in $H^{1}(\R^3)$.} Since $u_{n} \rightharpoonup u\not\equiv0$ in $H^{1}(\R^3)$, we get $B:=\mathop {\lim }\limits_{n  \to \infty}||\nabla u_{n}||_{2}^{2}\geq||\nabla u||_{2}^{2}>0$.
Then, (\ref{eq4.2}) implies that
\begin{equation} \label{eqa4.02}
 \Bigl(a+bB \Bigr)\int_{{\R^3}}\nabla u \nabla\varphi \!-\!\mu\int_{\mathbb{R}^{3}} \left| u \right|^{q\!-\!2} u {\varphi}\!-\! \int_{{\mathbb{R}^3}} \left|u\right|^{p\!-\!2} u {\varphi}\!-\!\lambda \int_{{\R^3}} u{\varphi}=0,~~~~~~~~\forall \varphi \in H^{1}(\R^3).
\end{equation}
Test (\ref{eq4.2})-(\ref{eqa4.02}) with $\varphi=u_{n}-u$, we obtain
$
(a+bB )\|\nabla(u_{n}-u)\|^{2}_2-\lambda \|u_{n}-u\|_2^{2}\to0$.
\end{proof}

The Sobolev critical case $q\!\in\!(2,6)$ and $p\!=\!6$ is more difficult than the case $p,q\!\in\!(2,6)$. We develop a perturbed Pohozaev constraint approach to prove Proposition \ref{prp4.2}. The key point is a revision of $P_{\mu}\left(u_{n}\right)\!=\!o_n(1)$, which makes it possible to split $P_{\mu}\left(u_{n}\right)\!=\!o_n(1)$ via the Br\'{e}zis-Lieb lemma (see \cite{a7}).

\begin{proposition} \label{prp4.2}
Let $a\!>\!0$, $b\!>\!0$, $c\!>\!0$, $\mu\!>\!0$, $2\!<\!q\!<\!\frac{14}{3}\!<\!p\!=\!6$ or $\frac{14}{3}\!<\!q\!<\!p\!=\!6$. Let $\left\{u_{n}\right\} \subset S_{c, r}$ be a Palais-Smale sequence for $\left.E_{\mu}\right|_{S_{c}}$ at energy level $m\not=0$, with
$$
m <\frac{a\mathcal{S}\Lambda}{3}\!+\!\frac{b\mathcal{S}^2{\Lambda}^2}{12} \quad \text { and } \quad P_{\mu}\left(u_{n}\right) \rightarrow 0 ~~~~\text { as } ~~~~ n \rightarrow \infty,
$$
where $\mathcal{S}\!=\!\inf _{v \in {{D}}^{1,2}(\mathbb{R}^{3})\!\setminus\! \{0\} }    \frac{\left\|\nabla v\right\|_{2}^{2}}{||v||_{6}^{2}}$ and $\Lambda=\frac{b{\mathcal{S}}^2}{2}
+\sqrt{a\mathcal{S}+\frac{b^2{\mathcal{S}}^4}{4}}$. Then, up to a subsequence, one of the following alternatives holds: \\
(i) either $u_{n} \!\rightharpoonup \!u\!\not \equiv \!0$ weakly in $H^1(\R^3)$ but not strongly, where $u $ solves
$$ -(a+B b)\Delta u
   =\lambda u+ {| u |^{4}}u+\mu {| u |^{q - 2}}u \text { in } \mathbb{R}^{3} \eqno (3.4)_{\lambda} $$
for some $\lambda\!<\!0$, and $m\!-\!(\frac{a\mathcal{S}\Lambda}{3}\!+\!\frac{b\mathcal{S}^2{\Lambda}^2}{12})
\!\geq\!I_{\mu}(u)\!:=\!(\frac{a}{2}\!+\! \frac{Bb}{4}){\| \nabla u \|}_2^2\!-\!\frac{1}{6}{\|u \|}_6^6\!-\!\frac{\mu}{q}{\| u \|}_q^q$ for $B\!:=\!\mathop {\lim }\limits_{n  \to \infty}||\nabla u_{n}||_{2}^{2}\!>\!0$.\\
(ii) or $u_{n} \rightarrow u$ strongly in $H^1(\R^3)$ for some $u\in H^1(\R^3)$. Moreover, $u\in S_c$, $E_{\mu}(u)=m$ and $u$ solves $(1.1)_{\lambda}$-(\ref{eq1.2}) for some $\lambda<0$.
\end{proposition}

\begin{proof}
The proof is divided into four main steps.
Similar to the proof of Proposition \ref{prp4.1}, we can easily get steps (1) and (2), that is,\\
\noindent \textbf{(1) $\{u_n\}$ is bounded in $H^1(\R^3)$ and $u_{n} \rightharpoonup u$ weakly in $H^{1}(\R^3)$ for some $u\in H^1(\R^3)$.}


\noindent \textbf{(2) $\exists$ Lagrange multipliers $\lambda_{n} \rightarrow \lambda \in \mathbb{R}$.}
Moreover, we have
\begin{equation} \label{equa3.6}
\Bigl(a+b{\| {\nabla u_n} \|}_2^2 \Bigr)\int_{{\R^3}}\nabla u_n \nabla\varphi \!-\!\mu\int_{\mathbb{R}^{3}} \left| u_n \right|^{q\!-\!2} u_n {\varphi}\!-\! \int_{{\mathbb{R}^3}} \left|u_n\right|^{4} u_n {\varphi}\!-\!\lambda_n \int_{{\R^3}} u_n{\varphi}=o_n(1)
\end{equation}
for every $\varphi \in H^1(\R^3)$, where $o_n(1) \rightarrow 0$ as $n \rightarrow \infty$. In particular, take $\varphi=u_n$, then
$$\lambda_{n} c^{2}=a{\| {\nabla u_n} \|}_2^2+b{\| {\nabla u_n} \|}_2^4-
\mu{||u_n||}_q^{q}-{||u_n||}_6^{6}+o_n(1).$$

\noindent \textbf{(3) $\lambda<0$ and $u\not\equiv0$.} Recalling that $P_{\mu}\left(u_{n}\right) \rightarrow 0$, we have
$$\lambda_{n} c^{2}=\mu(\delta_q-1)||u_{n}||_{q}^{q}+o_n(1).$$
Letting $n\to+\infty$, then $\lambda c^{2}=\mu(\delta_q-1)||u||_{q}^{q}$. Since $\mu>0$ and $0<\delta_q<1$, we deduce that $\lambda \leq 0$, with ``=" if and only if $u\equiv0$. If $\lambda_{n} \rightarrow 0$, we have
$$\mathop {\lim }\limits_{n  \to \infty} (a{\| {\nabla u_n} \|}_2^2+b{\| {\nabla u_n} \|}_2^4)
=\mathop {\lim }\limits_{n  \to \infty} {||u_n||}_6^6=\ell.$$
So $\mathop {\lim }\limits_{n  \to \infty} {\| {\nabla u_n} \|}_2^2=\sqrt{\frac{\ell}{b}+\frac{a^2}{4b^2}}-\frac{a}{2b}$ and by the Sobolev inequality $\ell\geq b{\mathcal{S}}^2 \ell^{\frac{2}{3}}+a\mathcal{S}\ell^{\frac{1}{3}}$. Since
$$0\not =m=\lim_{n\to+\infty} E_{\mu}\left(u_{n}\right)=\lim_{n\to+\infty}\Big[\frac{a}{2}{\| \nabla u_{n} \|}_2^2+\frac{b}{4}{\| \nabla u_{n} \|}_2^4- \frac{1}{6}{\|u_{n} \|}_6^6\Big]=\frac{\ell}{12}
+\frac{a}{4}\sqrt{\frac{\ell}{b}+\frac{a^2}{4b^2}}-\frac{a^2}{8b},$$
we get $\ell \not =0$ and $\ell \geq {\Lambda}^3$, where $\Lambda=\frac{b{\mathcal{S}}^2}{2}
+\sqrt{a\mathcal{S}+\frac{b^2{\mathcal{S}}^4}{4}}$. This leads to
 $$m=\mathop {\lim }\limits_{n  \to \infty} E_{\mu}\left(u_{n}\right)\geq\frac{{\Lambda}^3}{12}
+\frac{a}{4}\sqrt{\frac{{\Lambda}^3}{b}+\frac{a^2}{4b^2}}-\frac{a^2}{8b}
=\frac{{\Lambda}^3}{12}+\frac{a\mathcal{S}\Lambda}{4}
=\frac{a\mathcal{S}\Lambda}{3}\!+\!\frac{b\mathcal{S}^2{\Lambda}^2}{12},$$
which contradicts with our assumptions $m<\frac{a\mathcal{S}\Lambda}{3}\!+\!\frac{b\mathcal{S}^2{\Lambda}^2}{12}$. So, we have $\lambda <0$ and $u\not\equiv0$.

\noindent \textbf{(4) Conclusion.} Since $u_{n} \rightharpoonup u\not\equiv0$ in $H^{1}(\R^3)$, we get $B:=\mathop {\lim }\limits_{n  \to \infty}||\nabla u_{n}||_{2}^{2}\geq||\nabla u||_{2}^{2}>0$.
Then, (\ref{equa3.6}) implies that
\begin{equation} \label{eqa4.12}
 (a+B b)\int_{{\R^3}}\nabla u \nabla\varphi \!-\!\mu\int_{\mathbb{R}^{3}} \left| u \right|^{q\!-\!2} u {\varphi}\!-\! \int_{{\mathbb{R}^3}} \left|u\right|^{4} u {\varphi}\!-\!\lambda \int_{{\R^3}} u{\varphi}=0,~~~~~~~~\forall \varphi \in H^{1}(\R^3).
\end{equation}
That is, $u$ satisfies $-(a+B b)\Delta u=\lambda u+ {| u |^{4}}u+\mu {| u |^{q - 2}}u$. So we have the Pohozaev identity
$$Q_{\mu}(u):=(a+B b)||\nabla u||_{2}^2-\mu\delta_{q}{||u||}_q^q- {||u||}_6^6=0.$$

Denote $v_{n}=u_{n}-u$, then $ v_{n} \rightharpoonup 0$ in $H^{1}\left(\mathbb{R}^{3}\right)$ and $||\nabla u_n||_{2}^{2}=||\nabla  u||_{2}^{2}+||\nabla  v_n||_{2}^{2}+o_n(1)$. 
By the Br\'{e}zis-Lieb lemma in \cite{a7}, we have
$$
||u_n||_{6}^{6}=||u||_{6}^{6}+||v_n||_{6}^{6}+o_n(1),
~~~~~~~~||u_n||_q^q=||u||_q^q+||v_n||_q^q+o_n(1).
$$
Since $v_{n} \rightarrow 0$ strongly in $L^{q}(\R^3)$, we have $||u_n||_q^q\!=\!||u||_q^q\!+\!o_n(1)$. Rewrite $P_{\mu}\left(u_{n}\right)\!=\!o_n(1)$ as
$$P_{\mu}\left(u_{n}\right)= (a+B b)||\nabla u_{n}||_{2}^2-\mu\delta_{q}{||u||}_q^q- {||u_{n}||}_6^6+o_n(1).$$
From $Q_{\mu}\left(u\right)\!=\!0$, we have
$\ell\!=\!\mathop {\lim }\limits_{n  \to \infty} {||v_n||}_6^6\!=\!\mathop {\lim }\limits_{n  \to \infty} (a+B b)||\nabla v_{n}||_{2}^2\!\geq \!\mathop {\lim }\limits_{n  \to \infty} (a{\| {\nabla v_n} \|}_2^2+b{\| {\nabla v_n} \|}_2^4)$.
The Sobolev inequality implies that
\begin{align*}
\ell\geq a\mathcal{S}\ell^{\frac{1}{3}}+b{\mathcal{S}}^2 \ell^{\frac{2}{3}},~~~~~~~~\mathop {\lim }\limits_{n  \to \infty} (a{\| {\nabla v_n} \|}_2^2+b{\| {\nabla v_n} \|}_2^4)
\leq \mathop {\lim }\limits_{n  \to \infty} {||v_n||}_6^6\leq \frac{1}{{\mathcal{S}}^3}\mathop {\lim }\limits_{n  \to \infty} {\| {\nabla v_n} \|}_2^6.
\end{align*}
We get $\ell \!\geq \!{\Lambda}^3$ and $\mathop {\lim }\limits_{n  \to \infty} {\| {\nabla v_n} \|}_2^2 \!\geq\! \mathcal{S}{\Lambda}$ or $\ell\!=\!0\!=\!\mathop {\lim }\limits_{n  \to \infty} {\| {\nabla v_n} \|}_2^2$. Two possible cases may occur:

(i) $\ell \!\geq \!{\Lambda}^3$ and $\mathop {\lim }\limits_{n  \to \infty} {\| {\nabla v_n} \|}_2^2 \!\geq\! \mathcal{S}{\Lambda}$. Then, we have
\begin{align*}
m=\lim_{n\to+\infty} E_{\mu}\left(u_{n}\right)&= I_{\mu}(u)+\lim_{n\to+\infty} \Big[\frac{ a }{2}{\| \nabla v_n \|}_2^2+\frac{ Bb }{4}{\| \nabla v_n \|}_2^2- \frac{{\| v_n \|}_6^6}{6}\Big]\\
&= I_{\mu}(u)+ \frac{\ell}{12}+\lim_{n\to+\infty} \frac{ a }{4}{\| \nabla v_n \|}_2^2  \geq I_{\mu}(u)+\frac{a\mathcal{S}\Lambda}{3}\!+\!\frac{b\mathcal{S}^2{\Lambda}^2}{12},
\end{align*}
where $I_{\mu}(u):=(\frac{a}{2}+ \frac{Bb}{4}){\| \nabla u \|}_2^2- \frac{1}{6}{\|u \|}_6^6-\frac{\mu}{q}{\| u \|}_q^q $. In this case, alternative (i) follows.

(ii) $\ell\!=\!0$. Then $ u_{n} \!\rightarrow\! u$ in ${D}^{1,2}(\mathbb{R}^{3})$ and $L^{6}(\mathbb{R}^{3})$. Test \eqref{equa3.6}-(\ref{eqa4.12}) with $\varphi\!=\!u_{n}\!-\!u$, we have $(a+B b) \|\nabla(u_{n}-u)\|^{2}_2-\lambda \|u_{n}-u\|_2^{2}\to0$. In this case, alternative (ii) holds.
\end{proof}

\section{Mixed critical case} 
In this Section, we always assume that $2\!<\!q\!<\!\frac{10}{3}$ and $\frac{14}{3}\!<\!p\!\leq\!6$. Subsection 4.1 is devoted to locating the exact position of some critical points to $E_{\mu}|_{S_c}$. In Subsection 4.2, we prove Theorems \ref{th1.1}-\ref{th1.3}. Under the setting $2\!<\!q\!<\!\frac{10}{3}$ and $\frac{14}{3}\!<\!p\!\leq\!6$, $E_{\mu}|_{S_c}$ admits a convex-concave geometry if $0<\mu<\mu^{*}$, so we get a local minimizer and a Mountain Pass type critical point for $E_{\mu}|_{S_c}$ if $p\!<\!6$. When it comes to $2\!<\!q\!<\!\frac{10}{3}$ and $p\!=\!6$, we only obtain a local minimizer for $E_{\mu}|_{S_c}$.

\subsection{The exact location of some critical points to $E_{\mu}|_{S_c}$ for $2\!<\!q\!<\!\frac{10}{3}$ and $\frac{14}{3}\!<\!p\!\leq\!6$}
In this Subsection, we study the structure of $\mathcal{P}_{c,\mu}$ and $E_{\mu}$ to locate the position of critical points of $E_{\mu}|_{S_c}$. Since $2\!<\!q\!<\!\frac{10}{3}$ and $\frac{14}{3}\!<\!p\!\leq\!6$, we have $q\delta_q\!<\!2$ and $4\!<\!p\delta_p$. Let $\mathcal{C}_p$ be given by \eqref{equ2.2} for $p<6$, $\mathcal{C}_{p}={\mathcal{S}}^{-\frac{1}{2}}$ for $p=6$. Observing
$\mathcal{P}_{c, \mu}=\mathcal{P}_{+}^{c, \mu}\cup \mathcal{P}_{0}^{c, \mu} \cup \mathcal{P}_{-}^{c, \mu} $, we have:

\begin{lemma}\label{lem3.3}
Let $a>0$, $b>0$, $c>0$, $2\!<\!q\!<\!\frac{10}{3}$, $\frac{14}{3}\!<\!p\!\leq\!6$ and $0<\mu<\mu_{*}$.
Then $\mathcal{P}_{0}^{c, \mu}=\emptyset$ and $\mathcal{P}_{c, \mu}$ is a smooth manifold of codimension 2 in $H^{1}(\R^3)$. Here $\mu_{*}$ was defined in \eqref{ConAt1.10}.
\end{lemma}

\begin{proof}
Firstly, we claim that $\mathcal{P}_{0}^{c, \mu}=\emptyset$. Otherwise, there exists $u \in \mathcal{P}_{0}^{c, \mu}$. From $P_{\mu}(u)=0$ and $(\Psi_{u}^{\mu})^{\prime \prime}(0)=0$, we have
$$
a||\nabla u||_{2}^2+b||\nabla u||_{2}^4=\mu\delta_{q}{||u||}_q^q+\delta_{p} {||u||}_p^p,
~~~~~~~~2a||\nabla u||_{2}^2+4b||\nabla u||_{2}^4=\mu q \delta_{q}^2{||u||}_q^q+p\delta_{p}^2 {||u||}_p^p.
$$
By using \eqref{equ2.2}, we have
\begin{align*}
&(2-q\delta_q)a||\nabla u||_{2}^2+(4-q\delta_q)b||\nabla u||_{2}^4
=\delta_p(p\delta_p-q\delta_q)||u||_{p}^{p}\leq \delta_p(p\delta_p-q\delta_q) \mathcal{C}_p^{p} c^{p(1-\delta_p)} ||\nabla u||_{2}^{p\delta_p},\\
&(p\delta_p-2)a||\nabla u||_{2}^2+(p\delta_p-4)b||\nabla u||_{2}^4
=\mu\delta_q(p\delta_p-q\delta_q)||u||_{q}^{q}\leq  \mu\delta_q(p\delta_p-q\delta_q) \mathcal{C}_q^{q} c^{q(1-\delta_q)} ||\nabla u||_{2}^{q\delta_q}.
\end{align*}
Then, the lower and upper bounds of $||\nabla u||_{2}$ are given by
$$
 {\Big[ \frac{(4-q\delta_q)b}{\delta_p(p\delta_p-q\delta_q)\mathcal{C}_p^p c^{p(1-\delta_p)}} \Big]}^{\frac{1}{p\delta_p-4}} \leq ||\nabla u||_{2}\leq  {\Big[ \frac{\mu\delta_q(p\delta_p-q\delta_{q})\mathcal{C}_q^{q} c^{q(1-\delta_q)}}{(p\delta_{p}-4)b} \Big]}^{\frac{1}{4-q\delta_q}}.
$$
This leads to
$ \mu\geq \frac{(p\delta_{p}-4)b}{ \delta_q(p\delta_p-q\delta_{q})\mathcal{C}_q^{q} } {\Big[ \frac{(4-q\delta_q)b}{\delta_p(p\delta_p-q\delta_q)\mathcal{C}_p^p } \Big]}^{\frac{4-q\delta_q}{p\delta_p-4}}\frac{1}{c^{ q(1-\delta_q)+\frac{p(1-\delta_p)(4-q\delta_q)}{p\delta_p-4} }}>\mu_{*}$, which contradicts to $\mu<\mu_{*}$. Here $\mu_{*}$ was defined in \eqref{ConAt1.10}. We also used the fact that
$(\frac{p\delta_p}{4})^{4-q\delta_q}(\frac{q\delta_q}{4})^{p\delta_p-4}<1$ and this can be proved by using the monotonicity of $\frac{\ln x}{x-1}$.
Similar to the proof of Lemma 5.2 in \cite{NsAe}, we can check that $\mathcal{P}_{c, \mu}$ is a smooth manifold of codimension 2 in $H^{1}(\R^3)$. 
\end{proof}

Since $\mathcal{P}_{0}^{c, \mu}=\emptyset$, we get $\mathcal{P}_{c, \mu}=\mathcal{P}_{+}^{c, \mu}\cup \mathcal{P}_{-}^{c, \mu}$ with $\mathcal{P}_{+}^{c, \mu}\cap \mathcal{P}_{-}^{c, \mu}=\emptyset$. We can prove that $\mathcal{P}_{c, \mu}$ is a natural constraint in the following sense:
\begin{lemma}\label{lemma1.3}
Let $a>0$, $b>0$, $c>0$, $2\!<\!q\!<\!\frac{10}{3}$, $\frac{14}{3}\!<\!p\!\leq\!6$ and $0\!<\!\mu\!<\!\mu_{*}$. If $u\!\in\! \mathcal{P}_{c,\mu}$ is a critical point for $E_{\mu}|_{\mathcal{P}_{c,\mu}}$, then $u$ is a critical point for $E_{\mu}|_{S_c}$. Here $\mu_{*}$ was defined in \eqref{ConAt1.10}.
\end{lemma}

\begin{proof}
We only prove the case $p\in(\frac{14}{3},6)$. For the case $p=6$, the proof is much easier since $\delta_{p}=1$. We deduce by Lemma \ref{lem3.3} that $\mathcal{P}_{c, \mu}$ is a smooth manifold of codimension $2$ in $H^{1}$ and $\mathcal{P}_0^{c, \mu}=\emptyset$. If $u \in \mathcal{P}_{c,\mu}$ is a critical point for $E_{\mu}|_{\mathcal{P}_{c,\mu}}$, then by the Lagrange multipliers rule, there exists $\lambda, \nu \in \mathbb{R}$ such that
$$
\langle E_{\mu}'(u),\varphi \rangle-\lambda \int_{\mathbb{R}^{3}} u {\varphi}-\nu \langle  P_{\mu}'(u),\varphi\rangle=0,~~~~~~~~\forall \varphi \in H^1(\mathbb{R}^{3}).
$$
So $u$ solves $-\big[(1-2\nu)a +(1-4\nu)b||\nabla u||_{2}^2 \big] \Delta u-\lambda u+\mu (\nu q\delta_q-1) |u|^{q-2} u+(\nu p\delta_p-1)|u|^{p-2}u=0$. Combined with the Pohozaev identity, we have
$$
(1-2\nu)a||\nabla u||_{2}^2+(1-4\nu)b||\nabla u||_{2}^4+\mu\delta_{q}(\nu q\delta_q-1) {||u||}_q^q+\delta_{p}(\nu p\delta_p-1) {||u||}_p^p=0.
$$
Since $u \!\in \!\mathcal{P}_{c, \mu}$ and $u \!\notin\! \mathcal{P}_{0}^{c, \mu}$, we deduce from $\nu (2a||\nabla u||_{2}^2\!+\!4b||\nabla u||_{2}^4\!-\!\mu q \delta_{q}^2{||u||}_q^q\!-\!p\delta_{p}^2 {||u||}_p^p)\!=\!0$ that $\nu\!=\!0$.
\end{proof}

Next, we study the fiber map $\Psi_{u}^{\mu}(s)$ and determine the location and types of some critical points for $E_{\mu}|_{S_c}$.  Consider the constrained functional $E_{\mu}|_{S_c}$, by (\ref{equ2.2}), we have
\begin{align} \label{equ2.7}
 E_{\mu}(u) \!\geq\! \frac{a}{2}{\| \nabla u \|}_2^2\!+\!\frac{b}{4}{\| \nabla u \|}_2^4\!-\! \frac{ \mathcal{C}_p^{p}}{p}\left\|\nabla u\right\|_{2}^{p\delta_p} c^{p(1-\delta_p)}\!-\!\frac{ \mu \mathcal{C}_q^{q}}{q}\left\|\nabla u\right\|_{2}^{q\delta_q} c^{q(1-\delta_q)},~~~~ \forall u \in S_c.
\end{align}
To understand the geometry of $E_{\mu}|_{S_c}$, we introduce the function $h : \mathbb{R}^{+} \rightarrow \mathbb{R}$:
\begin{align} \label{eQua2.7}
h(t)=\frac{a}{2}t^2\!+\!\frac{b}{4}t^4\!-\! \frac{ \mathcal{C}_p^{p}}{p} c^{p(1-\delta_p)}t^{p\delta_p}\!-\!\frac{ \mu \mathcal{C}_q^{q}}{q} c^{q(1-\delta_q)}t^{q\delta_q}.
\end{align}
Since $\mu\!>\!0$, $q\delta_q\!<\!2$ and $4\!<\!p\delta_p$, we have that $h(0^+)=0^{-}$ and $h(+\infty)=-\infty $. If $p=6$, we have $\delta_p=1$, $\mathcal{C}_{p}={\mathcal{S}}^{-\frac{1}{2}}$ and hence
$h(t)=\frac{a}{2}t^2\!+\!\frac{b}{4}t^4\!-\!\frac{ \mu \mathcal{C}_q^{q}}{q} c^{q(1-\delta_q)}t^{q\delta_q}\!-\! \frac{ \mathcal{S}^{-3}}{6} t^6.$

\begin{lemma}  \label{LemA2.7}
Let $\tilde{a},\tilde{b},\tilde{c},\tilde{d},\tilde{p},\tilde{q} \in(0,+\infty)$ and   $f(t):=\tilde{a}t^2+\tilde{b}t^4-\tilde{c}t^{\tilde{p}}-\tilde{d}t^{\tilde{q}}$ for $t\geq0$. If $\tilde{p}\!\in\!(4,+\infty)$, $\tilde{q}\!\in\!(0,2)$ and $ \Big[ \Big( \frac{8(4-\tilde{q})}{ \tilde{p}(\tilde{p}-2)(\tilde{p}-\tilde{q}) } \Big)^{  \frac{4-\tilde{q}}{\tilde{p}-4} }\!-\! \Big( \frac{8(4-\tilde{q})}{ \tilde{p}(\tilde{p}-2)(\tilde{p}-\tilde{q}) } \Big)^{  \frac{\tilde{p}-\tilde{q}}{\tilde{p}-4} } \Big]\Big[\frac{ \tilde{a}}{\tilde{d}} \Big( \frac{\tilde{b}}{\tilde{c}} \Big)^{  \frac{2-\tilde{q}}{\tilde{p}-4} }\!+\!\frac{1}{\tilde{d}}\frac{{\tilde{b}}^{  \frac{\tilde{p}-\tilde{q}}{\tilde{p}-4} }}{ {\tilde{c}}^{  \frac{4-\tilde{q}}{\tilde{p}-4} } }\Big]\!>\!1$, then $f(t)$ has a local strict minimum at a negative level and a global strict maximum at
a positive level on $[0,+\infty)$.
\end{lemma}
\begin{proof}
Direct calculations give
\begin{align*}
&f'(t)=t^{\tilde{q}-1}g(t) ~~~~\mbox{for}~~~~ g(t)=2\tilde{a}t^{2-\tilde{q}}+4\tilde{b}t^{4-\tilde{q}}
-{\tilde{p}}\tilde{c}t^{\tilde{p}-\tilde{q}}
-{\tilde{q}}\tilde{d};  \\
&g'(t)=t^{1-\tilde{q}}w(t) ~~~~\mbox{for}~~~~ w(t)=2(2-\tilde{q})\tilde{a}+4(4-\tilde{q})\tilde{b}t^2
-{\tilde{p}}(\tilde{p}-\tilde{q})\tilde{c}t^{\tilde{p}-2};\\
&w'(t)=8(4-\tilde{q})\tilde{b}t -{\tilde{p}}(\tilde{p}-2)(\tilde{p}-\tilde{q})\tilde{c}t^{\tilde{p}-3}.
\end{align*}
Let $t^*=\Big( \frac{8(4-\tilde{q})\tilde{b}}{ \tilde{p}(\tilde{p}-2)(\tilde{p}-\tilde{q}) \tilde{c}} \Big)^{  \frac{1}{\tilde{p}-4} }$, then we have $w'(t)>0$ if $t\in(0,t^*)$ and $w'(t)<0$ if $t\in(t^*,+\infty)$. Consequently, $w(t)\nearrow$ on $[0,t^*)$ and $\searrow$ on $(t^*,+\infty)$. Since $w(0)>0$ and $w(+\infty)=-\infty$, $w(t)$ possesses unique zero point at some $\bar{t}$ with $\bar{t}>t^*$. So we have $g(t)\nearrow$ on $[0,\bar{t})$ and $\searrow$ on $(\bar{t},+\infty)$. We deduce from $\frac{A_2\!-\!A_3}{\tilde{d}} \Big[ \tilde{a}  \Big( \frac{\tilde{b}}{\tilde{c}} \Big)^{  \frac{2-\tilde{q}}{\tilde{p}-4} }\!+\! \frac{{\tilde{b}}^{  \frac{\tilde{p}-\tilde{q}}{\tilde{p}-4} }}{ {\tilde{c}}^{  \frac{4-\tilde{q}}{\tilde{p}-4} } }\Big]\!>\!1$ that
\begin{align*}
&\frac{2A_1 \tilde{a}}{\tilde{q} \tilde{d}}   \Big( \frac{\tilde{b}}{\tilde{c}} \Big)^{  \frac{2-\tilde{q}}{\tilde{p}-4} }\!+\!\frac{(4A_2\!-\!\tilde{p}A_3)}{\tilde{q}\tilde{d} }  \frac{{\tilde{b}}^{  \frac{\tilde{p}-\tilde{q}}{\tilde{p}-4} }}{ {\tilde{c}}^{  \frac{4-\tilde{q}}{\tilde{p}-4} } }\!>\!\frac{A_1 \tilde{a}}{ \tilde{d}}   \Big( \frac{\tilde{b}}{\tilde{c}} \Big)^{  \frac{2-\tilde{q}}{\tilde{p}-4} }+\frac{(A_2\!-\!A_3)}{ \tilde{d} }  \frac{{\tilde{b}}^{  \frac{\tilde{p}-\tilde{q}}{\tilde{p}-4} }}{ {\tilde{c}}^{  \frac{4-\tilde{q}}{\tilde{p}-4} } }
\!>\!1,
\end{align*}
where $A_1=\Big(\frac{8(4-\tilde{q})}{ \tilde{p}(\tilde{p}-2)(\tilde{p}-\tilde{q}) } \Big)^{  \frac{2-\tilde{q}}{\tilde{p}-4} }$, $A_2=\Big(\frac{8(4-\tilde{q})}{ \tilde{p}(\tilde{p}-2)(\tilde{p}-\tilde{q}) } \Big)^{  \frac{4-\tilde{q}}{\tilde{p}-4} }$ and $A_3=\Big( \frac{8(4-\tilde{q})}{ \tilde{p}(\tilde{p}-2)(\tilde{p}-\tilde{q}) } \Big)^{  \frac{\tilde{p}-\tilde{q}}{\tilde{p}-4} }$. This leads to $g(\bar{t})>g(t^*)>0$ and $f(t^*)>0$. Since $g(0)<0$, $g(\bar{t})>g(t^*)>0$ and $g(+\infty)=-\infty$, there exists unique $t_1, t_2$ ($0<t_1<t^*<\bar{t}<t_2$) such that $g(t_1)=0=g(t_2)$. Consequently, $f'(t)<0$ if $t\in(0,t_1)\cup(t_2,+\infty)$ and $f'(t)>0$ if $t\in(t_1,t_2)$. This implies that $f(t)\searrow$ on $[0,t_1)$, $\nearrow$ on $(t_1,t_2)$ and $\searrow$ on $(t_2,+\infty)$. The conclusion follows from $f(0)=0$, $f(t_2)>f(t^*)>0$ and $f(+\infty)=-\infty$.
\end{proof}

Similar to Lemma 5.1 and Lemma 5.3 in \cite{NsAe}, we can prove the following Lemmas \ref{lem3.1}-\ref{lem3.4}.

\begin{lemma}\label{lem3.1}
Let $a\!>\!0$, $b\!>\!0$, $c\!>\!0$, $2\!<\!q\!<\!\frac{10}{3}$, $\frac{14}{3}\!<\!p\!\leq\!6$ and $0\!<\!\mu\!<\!\mu^{*}$. Then the function $h$ has a local strict minimum at a negative level and a global strict maximum at a positive level. Moreover, there exist $0 < R_0 < R_1$, both depending
on $c$ and $\mu$, such that $h(R_0)=0=h(R_1)$ and $h(t)>0$ if and only if $t\in (R_0,R_1)$. Here $\mu^{*}$ was defined in \eqref{ConAt1.10}.
\end{lemma}

\begin{proof}
Take $\tilde{a}=\frac{a}{2}$, $\tilde{b}=\frac{b}{4}$, $\tilde{c}=\frac{ \mathcal{C}_p^{p}}{p} c^{p(1-\delta_p)}$, $\tilde{d}=\frac{\mu}{q} \mathcal{C}_q^{q} c^{q(1-\delta_q)}$, $\tilde{q}=q\delta_q$ and $\tilde{p}=p\delta_p$ in Lemma  \ref{LemA2.7}, then the conclusion follows provided $0\!<\!\mu\!<\!\mu^{*}$.
\end{proof}

\begin{lemma}\label{lem3.4}
Let $a\!>\!0$, $b\!>\!0$, $c\!>\!0$, $2\!<\!q\!<\!\frac{10}{3}$, $\frac{14}{3}\!<\!p\!\leq\!6$ and $0\!<\!\mu\!<\!\min\{\mu_{*},\mu^{*}\}$, where $\mu_{*},\mu^{*}$ were defined in \eqref{ConAt1.10}. For every $u \in S_{c}$, the function $\Psi_{u}^{\mu}$ has exactly two critical points $s_{u}<t_{u} \in \mathbb{R}$ and two zeros $c_{u}<d_{u} \in \mathbb{R}$, with $s_{u}<c_{u}<t_{u}<d_{u}$. Moreover:\\
$(1)$ $s_{u} \star u \in \mathcal{P}_{+}^{c,\mu}$ and $t_{u} \star u \in \mathcal{P}_{-}^{c,\mu}$, and if $s \star u \in \mathcal{P}_{c,\mu}$, then either $s=s_{u}$ or $s=t_{u};$ \\
$(2)$ $||\nabla (s \star u)||_{2} \leq R_{0}$ for every $s \leq c_{u}$, and
$$E_{\mu}\left(s_{u} \star u\right)=\min \left\{E_{\mu}(s \star u) : s \in \mathbb{R} \text { and }||\nabla(s \star u)||_{2}<R_{0}\right\}<0;$$
$(3)$ We have
$$ E_{\mu}\left(t_{u} \star u\right)=\max \{E_{\mu}(s \star u) : s \in \mathbb{R}\}>0,$$
and $\Psi^{\mu}_{u}$ is strictly decreasing on $\left(t_{u},+\infty\right);$  \\
$(4)$ The maps $u \in S_{c} \mapsto s_{u} \in \mathbb{R}$ and $u \in S_{c} \mapsto t_{u} \in \mathbb{R}$ are of class $C^1$.
\end{lemma}

\begin{proof}
Again we prove the case $p\!\in\!(\frac{14}{3},6)$. Letting $u\!\in \!S_c$,  then $u_t(x)=t^{\frac{3}{2}} u\left(t x\right)\in S_c$ for $t>0$. Consider the functional
$$ f(t)=E_{\mu}(u_t)=\frac{a}{2}t^{2} ||\nabla u||_{2}^2+\frac{b}{4}t^{4} ||\nabla u||_{2}^4-\mu \frac{t^{ q\delta_{q} }}{q} {||u||}_q^q-\frac{t^{ p\delta_{p}}}{p} {||u||}_p^p,~~~~~~~~\forall t\!>\!0$$
and take $\tilde{a}=\frac{a}{2}||\nabla u||_{2}^2$, $\tilde{b}=\frac{b}{4}||\nabla u||_{2}^4$, $\tilde{c}=\frac{1}{p} {||u||}_p^p$, $\tilde{d}=\frac{\mu}{q} {||u||}_q^q$, $\tilde{q}=q\delta_q$ and $\tilde{p}=p\delta_p$ in Lemma  \ref{LemA2.7}. By the following estimates
\begin{align*}
\frac{ ||\nabla u||_{2}^2 }{ ||u||_{q}^q } \Big[\frac{ ||\nabla u||_{2}^4 }{ ||u||_{p}^p } \Big]^{\frac{2-q\delta_q}{p\delta_p-4}}     \!\geq\! \frac{ ||\nabla u||_{2}^{2-q\delta_q} }{ \mathcal{C}_{q}^q   c^{q(1-\delta_q)} } \Big[\frac{ ||\nabla u||_{2}^{4-p\delta_p} }{ \mathcal{C}_{p}^p  c^{p(1-\delta_p)} } \Big]^{\frac{2-q\delta_q}{p\delta_p-4}}\!=\!\frac{ 1 }{ \mathcal{C}_{q}^q c^{q(1-\delta_q)} } \Big[\frac{ 1 }{ \mathcal{C}_{p}^p   c^{p(1-\delta_p)} } \Big]^{\frac{2-q\delta_q}{p\delta_p-4}}
\end{align*}
and
\begin{align*}
\frac{ 1 }{ ||u||_{q}^q } \frac{  \Big[ ||\nabla u||_{2}^4 \Big]^{\frac{p\delta_p-q\delta_q}{p\delta_p-4}} }{ \Big[ ||u||_{p}^p \Big]^{\frac{4-q\delta_q}{p\delta_p-4}} } \!\geq \! \frac{ 1 }{ \mathcal{C}_{q}^q  ||\nabla u||_{2}^{q\delta_q} c^{q(1-\delta_q)} } \frac{  ||\nabla u||_{2}^{q\delta_q} }{ \Big[ \mathcal{C}_{p}^p   c^{p(1-\delta_p)} \Big]^{\frac{4-q\delta_q}{p\delta_p-4}} }\!=\!\frac{ 1 }{ \mathcal{C}_{q}^q c^{q(1-\delta_q)} } \Big[\frac{ 1 }{ \mathcal{C}_{p}^p   c^{p(1-\delta_p)} } \Big]^{\frac{4-q\delta_q}{p\delta_p-4}},
\end{align*}
we deduce that $f(t)$ has a local strict minimum at a negative level and a global strict maximum at
a positive level on $[0,+\infty)$ provided $\mu\!<\!\mu^{*}$.
By monotonicity of composite functions, we derive that $\Psi_{u}^{\mu}(s):=E_{\mu}(s \star u)=f(e^s)$ has a local strict minimum at a negative level and a global strict maximum at
a positive level on $(-\infty,+\infty)$.

From (\ref{equ2.7}), we have
$$\Psi_{u}^{\mu}(s)=E_{\mu}(s \star u) \geq h\left(||\nabla(s \star u)||_{2}\right)=h\left(e^{s}||\nabla u||_{2}\right).$$
Thus, the $C^2$ function $\Psi_{u}^{\mu}$ is positive on $\left(\log \frac{R_{0}}{||\nabla u||_{2}}, \log \frac{R_{1}}{||\nabla u||_{2}}\right)$, and clearly $\Psi_{u}^{\mu}(-\infty)\!=\!0^{-}$, $\Psi_{u}^{\mu}(+\infty)\!=\!-\infty$. It follows that $\Psi_{u}^{\mu}$ has exactly two critical points $s_u\!<\!t_u$, with $s_u$ local minimum point on $(-\infty, \log \frac{R_{0}}{||\nabla u||_{2}})$ at negative level, and $t_u \!>\! s_u$ global maximum point at positive level.  By Corollary \ref{coroll2.1}, we have $s_{u} \star u$, $t_{u} \star u \in \mathcal{P}_{c,\mu}$, $s \star u \in \mathcal{P}_{c,\mu}$ implies $s \in\left\{s_{u}, t_{u}\right\}$. By minimality $(\Psi_{s_{u} \star u}^{\mu})^{\prime \prime}(0)=(\Psi_{u}^{\mu})^{\prime \prime}\left(s_{u}\right) \geq 0$, and ``=" can not hold, since $\mathcal{P}_{0}^{c,\mu}=\emptyset$; namely $s_{u} \star u \in \mathcal{P}_{+}^{c,\mu}$. Similarly, we have $t_{u} \star u \in \mathcal{P}_{-}^{c,\mu}$. By monotonicity and the behavior at infinity, $\Psi_{u}^{\mu}$ has exactly two zeros $c_{u}<d_{u}$, with $s_{u}<c_{u}<t_{u}<d_{u}$.

It remains to show that $u \mapsto s_{u}$ and $u \mapsto t_{u}$ are of class $C^1$. Consider the $C^1$ function $\Phi(s, u) :=(\Psi_{u}^{\mu})^{\prime}(s)$. By the facts that $\Phi\left(s_{u}, u\right)=0$, $\partial_{s} \Phi\left(s_{u}, u\right)>0$, and it is not possible to pass with continuity from $\mathcal{P}_{+}^{c,\mu}$ to $\mathcal{P}_{-}^{c,\mu}$ (since $\mathcal{P}_{0}^{c,\mu}=\emptyset$), then the implicit function theorem applied on $\Phi(s, u)$ gives the desired result. Similarly, we have $u \mapsto t_{u}$ is $C^{1}$.
\end{proof}

For $k>0$, let us set
$$ A_{k} :=\left\{u \in S_c :||\nabla u||_{2}<k\right\},~~~~\mbox{and}~~~~m(c,\mu):=\inf _{u \in A_{R_{0}}} E_{\mu}(u).$$

\begin{corollary} \label{coro3.5}
Let $a\!>\!0$, $b\!>\!0$, $c\!>\!0$, $2\!<\!q\!<\!\frac{10}{3}$, $\frac{14}{3}\!<\!p\!\leq\!6$ and $0\!<\!\mu\!<\!\min\{\mu_{*},\mu^{*}\}$, where $\mu_{*},\mu^{*}$ were defined in \eqref{ConAt1.10}. Then the set $\mathcal{P}_{+}^{c,\mu}$ is contained in $A_{R_{0}}=\left\{u \in S_c :||\nabla u||_{2}<R_{0}\right\}$, and $\sup _{\mathcal{P}_{+}^{c,\mu}} E_{\mu} \leq 0 \leq\inf _{\mathcal{P}_{-}^{c,\mu}} E_{\mu} $.
\end{corollary}

\begin{proof}
It is a direct conclusion of Lemma \ref{lem3.4}. Indeed, $\forall u \!\in \! \mathcal{P}_{+}^{c,\mu}$, Lemma \ref{lem3.4} implies that $s_{u}\!=\!0$, $E_{\mu}(u)\!\leq\!0$ and $||\nabla u||_{2}\!<\!R_{0}$. Similarly, $u \!\in\! \mathcal{P}_{-}^{c,\mu}$ implies that $t_{u}\!=\!0$ and $E_{\mu}(u)\!\geq\!0$.
\end{proof}

Let $\overline{A_{R_{0}}}$ be the closure of $A_{R_{0}}$ and $\overline{{A_{R_{0}}} } \!\setminus \!A_{R_{0}-\rho}\!=\!\{ u\!\in\! \overline{{A_{R_{0}}} }\!: u\!\not \in \!A_{R_{0}-\rho} \}$ for some $R_{0}$ and $\rho$.
\begin{lemma} \label{lem3.6}
Let $a\!>\!0$, $b\!>\!0$, $c\!>\!0$, $2\!<\!q\!<\!\frac{10}{3}$, $\frac{14}{3}\!<\!p\!\leq\!6$ and $0\!<\!\mu\!<\!\min\{\mu_{*},\mu^{*}\}$, where $\mu_{*},\mu^{*}$ were defined in \eqref{ConAt1.10}. It holds that $m(c, \mu) \in(-\infty, 0)$ and
$$m(c, \mu)=\inf _{\mathcal{P}_{c,\mu}} E_{\mu}=\inf _{\mathcal{P}_{+}^{c,\mu}} E_{\mu}.$$
Moreover, there exists a constant $\rho>0$ (independent of $c$ and $\mu$) small enough such that
$$
m(c, \mu)<\inf_{ \overline{{A_{R_{0}}} } \setminus A_{R_{0}-\rho}} E_{\mu}.
$$
\end{lemma}

\begin{proof}
For $u \in A_{R_{0}}$, we have $
E_{\mu}(u) \geq h\left(||\nabla u||_{2}\right) \geq \min _{t \in\left[0, R_{0}\right]} h(t)>-\infty$, and hence $m(c, \mu)>-\infty$. Moreover, for any $u\in S_c$ we have $||\nabla (s \star u)||_{2}<R_{0}$ and $E_{\mu}(s \star u)<0$ for $s \ll-1$, and hence $m(c, \mu)<0$.

By Corollary \ref{coro3.5}, we have $m(c, \mu) \leq \inf _{\mathcal{P}_{+}^{c,\mu}} E_{\mu}$ since $\mathcal{P}_{+}^{c,\mu} \subset A_{R_{0}}$. On the other hand, if $u \in A_{R_{0}}$, we have $s_{u} \star u \in \mathcal{P}_{+}^{c,\mu} \subset A_{R_{0}}$ and
$$
E_{\mu}\left(s_{u} \star u\right)=\min \left\{E_{\mu}(s \star u) : s \in \mathbb{R} \text { and }||\nabla(s \star u)||_{2}<R_{0}\right\} \leq E_{\mu}(u),
$$
which implies that $\inf _{\mathcal{P}_{+}^{c,\mu}} E_{\mu} \leq m(c, \mu)$. To prove that $\inf _{\mathcal{P}_{+}^{c,\mu}} E_{\mu}=\inf _{\mathcal{P}_{c,\mu}} E_{\mu}$, it is sufficient to recall that $E_{\mu}\geq0$ on $\mathcal{P}_{-}^{c,\mu}$, see Corollary \ref{coro3.5}.

Finally, by continuity of $h$ there exists $\rho> 0$ (independent of $c$ and $\mu$) such that $h(t)\geq \frac{m(c, \mu)}{2}$ if $t \in [R_0- \rho, R_0]$.
Therefore
$
E_{\mu}(u) \geq h\left(||\nabla u||_{2}\right) \geq \frac{m(c, \mu)}{2}>m(c, \mu)
$
for every $u\in S_c$ with $ ||\nabla u||_{2} \in [R_0- \rho, R_0]$.
\end{proof}

\begin{lemma} \label{lem3.7}
Let $a\!>\!0$, $b\!>\!0$, $c\!>\!0$, $2\!<\!q\!<\!\frac{10}{3}$, $\frac{14}{3}\!<\!p\!\leq\!6$ and $0\!<\!\mu\!<\!\min\{\mu_{*},\mu^{*}\}$. Suppose that $E_{\mu}(u)<m(c, \mu)$. Then the value $t_u$ defined by Lemma \ref{lem3.4} is negative. Here $\mu_{*},\mu^{*}$ were defined in \eqref{ConAt1.10}.
\end{lemma}

\begin{proof}
Let $s_{u}<c_{u}<t_{u}<d_{u}$ be defined by Lemma \ref{lem3.4}. If $d_{u} \leq 0$, then $t_{u}<0$, and hence we can assume by contradiction that $d_u>0$. If $0\in(c_{u}, d_{u})$, then $E_{\mu}(u)=\Psi^{\mu}_{u}(0)>0$, which is impossible since $E_{\mu}(u)<m(c, \mu)<0$. Therefore $c_{u}>0$, and by
Lemma \ref{lem3.4}-(2)
\begin{align*}
m(c, \mu) &>E_{\mu}(u)=\Psi_{u}^{\mu}(0) \geq \inf _{s \in\left(-\infty, c_{u}\right]} \Psi_{u}^{\mu}(s) \\ & \geq \inf \left\{E_{\mu}(s \star u) : s \in \mathbb{R} \text { and }||\nabla (s \star u)||_{2}<R_{0}\right\}=E_{\mu}\left(s_{u} \star u\right) \geq m(c, \mu) \end{align*}
which is again a contradiction.
\end{proof}

\begin{lemma} \label{lem3.8}
Let $a\!>\!0$, $b\!>\!0$, $c\!>\!0$, $2\!<\!q\!<\!\frac{10}{3}$, $\frac{14}{3}\!<\!p\!\leq\!6$ and $0\!<\!\mu\!<\!\min\{\mu_{*},\mu^{*}\}$, where $\mu_{*},\mu^{*}$ were defined in \eqref{ConAt1.10}. It holds that
$$\tilde{\sigma}(c, \mu) :=\inf _{u \in \mathcal{P}_{-}^{c,\mu}} E_{\mu}(u)>0.$$
\end{lemma}

\begin{proof}
Let $t_{max}$ be the strict maximum of the function $h$ at positive level, see Lemma \ref{lem3.1}. For every $u \in \mathcal{P}_{-}^{c,\mu}$, there exists $\tau_{u} \in \mathbb{R}$ such that $||\nabla \left(\tau_{u} \star u\right)||_{2}=t_{\max }$. Moreover, since $u \in \mathcal{P}_{-}^{c,\mu}$ we also have by Lemma \ref{lem3.4} that the value $0$ is the unique strict maximum of the function $\Psi_{u}^{\mu}$. Therefore
$$
E_{\mu}(u)=\Psi_{u}^{\mu}(0) \geq \Psi_{u}^{\mu}\left(\tau_{u}\right)=E_{\mu}\left(\tau_{u} \star u\right) \geq h\left(||\nabla\left(\tau_{u} \star u\right)||_{2}\right)=h\left(t_{\max }\right)>0.
$$
Since $u \in \mathcal{P}_{-}^{c,\mu}$ was arbitrarily chosen, we deduce that $\inf_{\mathcal{P}_{-}^{c,\mu}} E_{\mu} \geq \max _{\mathbb{R}} h>0$.
\end{proof}

\subsection{The existence and asymptotic results for $2\!<\!q\!<\!\frac{10}{3}$ and $\frac{14}{3}\!<\!p\!\leq\!6$}

In this Subsection, we first prove the existence results, i.e. Theorem \ref{th1.1}-(1)(2)(3) and Theorem \ref{th1.3}-(1)(2). The proof of Theorem \ref{th1.1} is divided into two parts. To begin with, we prove the existence of a local minimizer for $\left.E_{\mu}\right|_{S_{c}}$. Next, we construct a Mountain Pass type critical point for $\left.E_{\mu}\right|_{S_{c}}$. Finally, we prove the asymptotic results, i.e. Theorem \ref{th1.1}-(4)(5) and Theorem \ref{th1.3}-(3).\\

\noindent \textbf{Proof of Theorem \ref{th1.1}-(1),(2),(3):}

\noindent \textbf{(i) Existence of a local minimizer.}

Let $\{v_n\}$ be a minimizing sequence for $m(c,\mu):=\inf _{u \in A_{R_{0}}} E_{\mu}(u)$. From Section 3.3 and Lemma 7.17 in \cite{ELMA}, we have $E_{\mu}\left({|v_n|}^*\right)\!\leq\! E_{\mu}\left(v_{n}\right)$, since
\begin{align} \label{alig5.1}
\|\nabla{|v_n|}^*\|_2  \leq  \|\nabla{|v_n|} \|_2 ,~~~~||v_n||_p=||{|v_n|}^*||_p,~~~~||v_n||_q=||{|v_n|}^*||_q,
\end{align}
where ${|v_n|}^*$ is the symmetric decreasing rearrangement of $|v_n|$. So we can assume that $v_n\!\in\! S_c$ is nonnegative and radially decreasing for every $n$. By using Lemma \ref{lem3.4} and Corollary \ref{coro3.5}, we have $s_{v_{n}} \!\star\! v_{n} \!\in\! \mathcal{P}_{+}^{c,\mu}$, $||\nabla\left(s_{v_{n}} \!\star \!v_{n}\right)||_{2}\!<\!R_{0}$ and that
$$E_{\mu}\left(s_{v_{n}} \star v_{n}\right)=\min \left\{E_{\mu}(s \star v_{n}) : s \in \mathbb{R} \text { and }||\nabla(s \star v_{n})||_{2}<R_{0}\right\} \leq E_{\mu}\left(v_{n}\right).$$
Consequently, we obtain a new minimizing sequence $\left\{w_{n}=s_{v_{n}} \star v_{n}\right\}$ for $m(c,\mu)$, with
$$w_{n} \in S_{c,r} \cap \mathcal{P}_{+}^{c,\mu}~~~~\mbox{and}~~~~P_{\mu}(w_{n})=0$$
for every $n$. By Lemma \ref{lem3.6}, we have $||\nabla w_{n}||_2<R_0-\rho$ for every $n$. Hence,  the Ekeland's variational principle guarantees the existence of a new minimizing sequence $\left\{u_{n}\right\} \subset A_{R_{0}}$ for $m(c, \mu)<0$, with the property that $||u_{n}-w_{n}||_{H^1} \rightarrow 0$ as $n \rightarrow +\infty$, which is also a Palais-Smale sequence for $E_{\mu}$ on $S_{c}$. The condition $||u_{n}-w_{n}||_{H^1} \rightarrow 0$ implies
$$||\nabla u_{n}||_2 \leq R_0-\rho~~~~\mbox{and}~~~~P_{\mu}(u_{n}) \rightarrow 0~~~~\mbox{as}~~~~n \rightarrow \infty$$
and hence $\{u_n\}$ satisfies all the assumptions of Proposition \ref{prp4.1}. Therefore, up to a subsequence $u_{n} \rightarrow \tilde{u}_{\mu}$ strongly in $H^1$, $\tilde{u}_{\mu}$ is an interior local minimizer for $\left.E_{\mu}\right|_{A_{R_{0}}}$, and solves $(1.1)_{\tilde{\lambda}}$ for some $\tilde{\lambda}<0$. It is easy to know that $\tilde{u}_{\mu}$ is nonnegative and radially deceasing. The strong maximum principle implies that $\tilde{u}_{\mu}>0$.

Since any critical point of
$E_{\mu}|_{S_c}$ lies in $\mathcal{P}_{c,\mu}$ and $m(c, \mu)=\inf _{\mathcal{P}_{c,\mu}} E_{\mu}$ (see Lemma \ref{lem3.6}), we see that $\tilde{u}_{\mu}$ is a ground state for $\left.E_{\mu}\right|_{S_c}$. It only remains to prove that any ground state of $E_{\mu}|_{S_c}$ is a local minimizer of $E_{\mu}$ in $A_{R_0}$. Let then $u$ be a critical point of $E_{\mu}|_{S_c}$ with $E_{\mu}(u)=m(c, \mu)=\inf_{\mathcal{P}_{c,\mu}} E_{\mu}$. Since $E_{\mu}(u)<0<\inf_{\mathcal{P}^{c,\mu}_{-}} E_{\mu}$, necessarily $u \in \mathcal{P}^{c,\mu}_{+}$. Then Corollary \ref{coro3.5} implies that $\mathcal{P}^{c,\mu}_{+}\subset A_{R_0}$. This leads to $||\nabla u||_2<R_0$, and as a consequence $u$ is a local minimizer for $E_{\mu}|_{A_{R_0}}$.

\noindent \textbf{(ii) Existence of a Mountain Pass type solution.}

We focus now on the existence of a second critical point for $\left.E_{\mu}\right|_{S_c}$. Denote $E_{\mu}^{m}=\{u \in S_c : E_{\mu}(u) \leq m\}$. Motivated by \cite{NeJl}, we define the augmented functional $\tilde{E_{\mu}}: \mathbb{R} \times  H^1 \rightarrow \mathbb{R}$
$$\tilde{E_{\mu}}(s,u):=E_{\mu}(s \star u)=\frac{a}{2}e^{2s} ||\nabla u||_{2}^2+\frac{b}{4}e^{4s} ||\nabla u||_{2}^4-\mu \frac{e^{ q\delta_{q} s}}{q} {||u||}_q^q-\frac{e^{ p\delta_{p}s}}{p} {||u||}_p^p$$
and study $\tilde{E_{\mu}}|_{\mathbb{R}\times {S_c} }$. Notice that $S_{c,r}\!=\!H_{\mathrm{rad}}^{1} \!\cap\! S_c$ and $\tilde{E_{\mu}}$ is of class $C^1$. Theorem 1.28 in \cite{Mwlm} indicates that a critical point for $\tilde{E_{\mu}}|_{\mathbb{R}\times {S_{c,r}} }$ is a critical point for $\tilde{E_{\mu}}|_{\mathbb{R}\times {S_c} }$.

We introduce the minimax class
$$
\Gamma :=\left\{\gamma(\tau)=\big(\zeta(\tau),\beta(\tau)\big) \in C\left([0,1], \mathbb{R}\times S_{c,r}\right) ; \gamma(0) \in (0,\mathcal{P}_{+}^{c,\mu}), \gamma(1) \in (0, E_{\mu}^{2m(c, \mu)})\right\},
$$
then $\Gamma\not=\emptyset$. Indeed, $\forall u \in S_{c,r}$, by Lemma \ref{lem3.4} we know that there exists $s_{1} \gg 1$ such that
\begin{equation} \label{equa4.1}
\gamma_{u} : \tau \in[0,1] \mapsto   \big(0,\left((1-\tau) s_{u}+\tau s_{1}\right) \star u \big ) \in \mathbb{R}\times S_{c,r}
\end{equation}
is a path in $\Gamma$ (recall that $s \in \mathbb{R} \mapsto s \star u \in S_{c,r}$ is continuous, $s_{u} \star u \in \mathcal{P}_{+}^{c,\mu}$ and  $E_{\mu}(s \star u) \rightarrow-\infty$ as $s \rightarrow +\infty$). Thus, the minimax value
$$
\sigma(c, \mu) :=\inf _{\gamma \in \Gamma} \max _{(s,u) \in \gamma([0,1])} \tilde{E}_{\mu}(s,u)
$$
is a real number. We claim that
\begin{equation} \label{eq5.1}
\forall \gamma \in \Gamma~~~~\mbox{there exists}~~~~\tau_{\gamma} \in (0,1)~~~~\mbox{such that}~~~~\zeta(\tau_{\gamma}) \star \beta(\tau_{\gamma}) \in \mathcal{P}_{-}^{c,\mu}.
\end{equation}
Indeed, since $\gamma(0)\!=\!\big(\zeta(0),\beta(0)\big) \!\in\! (0,\mathcal{P}_{+}^{c,\mu})$, by Corollary \ref{coroll2.1} and Lemma \ref{lem3.4}, we have $t_{\zeta(0)\star \beta(0)}\!=\!t_{\beta(0)}\!>\!s_{\beta(0)}\!=\!0$; since $E_{\mu}(\beta(1))\!=\!\tilde{E}_{\mu}(\gamma(1)) \!\leq\! 2 m(c, \mu)$, by Lemma \ref{lem3.7}, we have
$$t_{\zeta(1)\star \beta(1)}=t_{\beta(1)}<0,$$
and moreover the map $t_{\zeta(\tau)\star \beta(\tau)}$ is continuous in $\tau$ (we refer again to Lemma \ref{lem3.4} and recall that $s \in \mathbb{R} \mapsto s \star u \in S_{c,r}$ is continuous). It follows that for every $\gamma \in \Gamma$ there exists $\tau_{\gamma} \in(0,1)$ such that $t_{\zeta(\tau_{\gamma})\star \beta(\tau_{\gamma})}=0$, and so $\zeta(\tau_{\gamma}) \star \beta(\tau_{\gamma}) \in \mathcal{P}_{-}^{c,\mu}$. Thus (\ref{eq5.1}) holds.

For every $\gamma \in \Gamma$, by (\ref{eq5.1}) we have
\begin{equation} \label{eqa6.3}
\max _{\gamma([0,1])} \tilde{E}_{\mu} \geq \tilde{E}_{\mu} \left(\gamma\left(\tau_{\gamma}\right)\right)={E}_{\mu}(\zeta(\tau_{\gamma}) \star \beta(\tau_{\gamma})) \geq \inf _{\mathcal{P}^{c,\mu}_{-}\cap {S_{c,r}}} E_{\mu},
\end{equation}
which gives $\sigma(c, \mu) \geq \inf _{\mathcal{P}_{-}^{c,\mu} \cap S_{c,r}} E_{\mu}$. On the other hand, if $u \in \mathcal{P}^{c,\mu}_{-} \cap S_{c,r},$ then $\gamma_{u}$ defined in (\ref{equa4.1}) is a path in $\Gamma$ with
$$
E_{\mu}(u)=\tilde{E}_{\mu}(0,u)=\max _{\gamma_u([0,1])} \tilde{E}_{\mu} \geq \sigma(c, \mu),
$$
which gives $ \inf _{\mathcal{P}^{c,\mu}_{-} \cap S_{c,r}} E_{\mu} \geq \sigma(c, \mu)$. This, Corollary \ref{coro3.5} and Lemma \ref{lem3.8} imply that
\begin{equation}  \label{eq5.3}
\sigma(c, \mu)=\inf _{\mathcal{P}^{c,\mu}_{-} \cap S_{c,r}} E_{\mu}>0 \geq \sup _{\left(\mathcal{P}^{c,\mu}_{+} \cup E_{\mu}^{2 m(c, \mu)}\right) \cap S_{c,r}} E_{\mu}=\sup _{\left((0,\mathcal{P}^{c,\mu}_{+}) \cup (0,E_{\mu}^{2 m(c, \mu)})\right) \cap  (\mathbb{R}\times S_{c,r})}  \tilde{E}_{\mu}.
\end{equation}

Let $\gamma_n(\tau)=\big(\zeta_n(\tau),\beta_n(\tau)\big)$ be any minimizing sequence for $\sigma(c, \mu)$ with the property that $\zeta_n(\tau)\equiv 0$ and $\beta_n(\tau)\geq0$ a.e. in ${\R}^3$ for every $\tau \in[0, 1]$ (Notice that, if $\{\gamma_n=\big(\zeta_n,\beta_n\big)\}$ is a minimizing sequence, then also $\{(0,\zeta_n\star {|\beta_n|})\}$ has the same property). Take
$$X=\mathbb{R}\times S_{c,r},~~~~\mathcal{F}=\{\gamma([0,1]):~~\gamma \in \Gamma\},~~~~B=(0,\mathcal{P}_{+}^{c,\mu}) \cup (0, E_{\mu}^{2m(c, \mu)}),$$
$$F=\{(s,u)\in\mathbb{R}\times S_{c,r} ~~~~|~~~~\tilde{E}_{\mu}(s,u)\geq \sigma(c, \mu) \},~~~~A=\gamma([0,1]),~~~~A_n=\gamma_n([0,1])$$
in Lemma \ref{lem3.10}. We need to checked that $\mathcal{F}$ is a homotopy stable family of compact subsets of $X$ with extended closed boundary $B$, and that $F$ is a dual set for $\mathcal{F}$, in the sense that assumptions (1) and (2) in Lemma \ref{lem3.10} are satisfied.

Indeed, since $\sigma(c, \mu)=\inf _{\mathcal{P}^{c,\mu}_{-} \cap S_{c,r}} E_{\mu}$, (\ref{eqa6.3}) $\Rightarrow \gamma\left(\tau_{\gamma}\right)=(\zeta(\tau_{\gamma}), \beta(\tau_{\gamma}))\in A \cap F$,  (\ref{eq5.3}) $\Rightarrow F \cap B=\emptyset$ and (2) in Lemma \ref{lem3.10}, then $A \cap F\not=\emptyset$ and $F \cap B=\emptyset$ give (1) in Lemma \ref{lem3.10}. For every $ \gamma \in \Gamma$, since $\gamma(0) \in (0,\mathcal{P}_{+}^{c,\mu})$ and $\gamma(1) \in (0, E_{\mu}^{2m(c, \mu)})$, we have $\gamma(0), \gamma(1) \in B$. Then for any set $A$ in $\mathcal{F}$ and any $\eta\in C([0,1]\times X;X)$
satisfying $\eta(t,x)=x$ for all $(t,x)\in (\{0\}\times X)\cup ([0,1]\times B)$, it holds that $\eta(1,\gamma(0))=\gamma(0),~~~~\eta(1,\gamma(1))=\gamma(1)$. So we have $\eta(\{1\}\times A)\in \mathcal{F}$.

Consequently, by Lemma \ref{lem3.10}, there exists a Palais-Smale sequence $\{(s_n,w_n)\}\subset \mathbb{R}\times S_{c,r}$ for $\tilde{E_{\mu}}|_{\mathbb{R}\times {S_{c,r}} }$ at level $\sigma(c, \mu)>0$ such that
\begin{equation}  \label{eq5.4}
\partial_{s} \tilde{E_{\mu}}\left(s_{n}, w_{n}\right) \rightarrow 0 \quad \text { and } \quad\left\|\partial_{u} \tilde{E_{\mu}}\left(s_{n}, w_{n}\right)\right\|_{\left(T_{w_{n}} S_{c,r}\right)^{*}} \rightarrow 0 \quad \text { as } n \rightarrow \infty,
\end{equation}
with the additional property that
\begin{equation}   \label{equq5.5}
\left|s_{n}\right|+\operatorname{dist}_{H^1}\left(w_{n}, \beta_{n}([0,1])\right) \rightarrow 0 \quad \text { as } n \rightarrow \infty.
\end{equation}
From (\ref{eq5.4}), we have $P_{\mu}\left(s_n \star w_{n}\right) \rightarrow 0$ and that
\begin{equation} \label{equa5.6}
\begin{gathered}
a e^{2 s_n}  \int_{\mathbb{R}^{3}} \nabla w_{n} \nabla {\varphi}+b e^{4 s_n}  \| \nabla w_{n}\|_2^2   \int_{\mathbb{R}^{3}} \nabla w_{n} \nabla {\varphi}-\mu e^{ q\delta_{q} s_n}\int_{\mathbb{R}^{3}} \left|w_{n}\right|^{q\!-\!2} w_{n} {\varphi} \\
-e^{ p\delta_{p} s_n}\int_{\mathbb{R}^{3}} \left|w_{n}\right|^{p\!-\!2} w_{n} {\varphi}=o(1)\|\varphi\|_{H^1},~~~~~~~~\forall \varphi \in T_{w_{n}} S_{c,r}.
\end{gathered}
\end{equation}
By using (\ref{equq5.5}), we know that ${s_n}$ is bounded from above and from below. Consequently,
\begin{equation} \label{equa5.10}
\langle E_{\mu}'\left(s_{n} \star w_{n}\right), s_{n} \star \varphi \rangle=o(1)\|\varphi\|_{H^1}=o(1)\left\|s_{n} \star \varphi\right\|_{H^1}~~\text{as}~~n \rightarrow \infty, \forall \varphi \in T_{w_{n}} S_{c,r}.
\end{equation}
From (\ref{equa5.10}) and Lemma \ref{lem2.18}, we see that $\{u_n:=s_n\star w_n\} \subset S_{c,r}$ is a Palais-Smale sequence for $E_{\mu}|_{S_{c,r}}$ at level $\sigma(c,\mu)>0$, with $P_{\mu}(u_n)\to0$. Therefore, all the assumptions of Proposition \ref{prp4.1} are satisfied, and we deduce that up to a subsequence $u_{n} \rightarrow \hat{u}_{\mu}$ strongly in $H^{1}$, with $\hat{u}_{\mu} \in S_{c,r}$ nonnegative radial solution to $(1.1)_{\hat{\lambda}}$ for some $\hat{\lambda}<0$. The strong maximum principle implies that $\hat{u}_{\mu}>0$.
\qed

\noindent \textbf{Proof of Theorem \ref{th1.3}-(1),(2):} \\
Imitating the proof of Theorem \ref{th1.1}-(1),
we get a Palais-Smale sequence $\{u_n\}$ for $E_{\mu}|_{S_{c}}$ with
$$||\nabla u_{n}||_2 \leq R_0-\rho~~~~\mbox{and}~~~~P_{\mu}(u_{n}) \rightarrow 0~~~~\mbox{as}~~~~n \rightarrow \infty$$
and $u_n$ is nonnegative and radially decreasing for every $n$. Hence $\{u_n\}$ satisfies all the assumptions of Proposition \ref{prp4.2}. We show that alternative (ii) in Proposition \ref{prp4.2} occurs. Otherwise, up to a subsequence $u_{n} \rightharpoonup \tilde{u}_{\mu}\not \equiv 0$ weakly in $H^1(\R^3)$ but not strongly, where $\tilde{u}_{\mu} $ is a solution to $(3.4)_{\tilde{\lambda}}$ for some $\tilde{\lambda}<0$, and
$$
I_{\mu}(\tilde{u}_{\mu}):= \Big(\frac{a}{2}+\frac{Bb}{4}\Big){\| \nabla \tilde{u}_{\mu} \|}_2^2- \frac{1}{6}{\|\tilde{u}_{\mu}\|}_6^6-\frac{\mu}{q}{\| \tilde{u}_{\mu} \|}_q^q \leq m(c, \mu)-\frac{a\mathcal{S}\Lambda}{3}\!-\!\frac{b\mathcal{S}^2{\Lambda}^2}{12},
$$
where $B:=\mathop {\lim }\limits_{n  \to \infty}||\nabla u_{n}||_{2}^{2}\geq{\| \nabla \tilde{u}_{\mu} \|}_2^2>0$ and $\Lambda=\frac{b{\mathcal{S}}^2}{2}
+\sqrt{a\mathcal{S}+\frac{b^2{\mathcal{S}}^4}{4}}$. Since $\tilde{u}_{\mu}$ solves $(3.4)_{\tilde{\lambda}}$, we get the Pohozaev identity
$Q_{\mu}(\tilde{u}_{\mu}):=(a+Bb)||\nabla \tilde{u}_{\mu}||_{2}^2-\mu\delta_{q}{||\tilde{u}_{\mu}||}_q^q- {||\tilde{u}_{\mu}||}_6^6=0$.
By using $||\tilde{u}_{\mu}||_2 \leq c$ and $I_{\mu}(\tilde{u}_{\mu})
=\frac{a}{3}{\| \nabla \tilde{u}_{\mu} \|}_2^2+\frac{Bb}{12}{\| \nabla \tilde{u}_{\mu} \|}_2^2-\mu
\left(\frac{1}{q}-\frac{\delta_q}{6}\right){||\tilde{u}_{\mu}||}_q^q$, we have
\begin{align}  \label{eqa5.5}
 m(c, \mu) & 
 \geq\frac{a\mathcal{S}\Lambda}{3}\!+\!\frac{b\mathcal{S}^2{\Lambda}^2}{12}+\frac{a}{3}{\| \nabla \tilde{u}_{\mu} \|}_2^2+\frac{Bb}{12}{\| \nabla \tilde{u}_{\mu} \|}_2^2-\mu
\left(\frac{1}{q}-\frac{\delta_q}{6}\right){||\tilde{u}_{\mu}||}_q^q \nonumber \\
& \geq \frac{a\mathcal{S}\Lambda}{3}\!+\!\frac{b\mathcal{S}^2{\Lambda}^2}{12}
+\frac{b}{12}{\| \nabla \tilde{u}_{\mu} \|}_2^4-\mu
\left(\frac{1}{q}-\frac{\delta_q}{6}\right)\mathcal{C}_q^{q}c^{q(1-\delta_q)}{\| \nabla \tilde{u}_{\mu} \|}_2^{q\delta_q}.
\end{align}
Denote $g(t)\!=\!\frac{b}{12}t^4-\mu
\left(\frac{1}{q}\!-\!\frac{\delta_q}{6}\right)\mathcal{C}_q^{q}c^{q(1-\delta_q)}t^{q\delta_q},~~~~\forall t\geq0$. By using $\mu\!<\!\mu^{**}$, we have
$\min _{t\geq0} g(t)=-\frac{b}{3}(\frac{1}{q\delta_q}-\frac{1}{4})t_0^4
>-\frac{a\mathcal{S}\Lambda}{3}\!-\!\frac{b\mathcal{S}^2{\Lambda}^2}{12}$ for $t_0={\Big[ \frac{\mu \delta_q(6-q\delta_q)\mathcal{C}_{q}^{q}c^{q(1-\delta_q)}}{2b} \Big]}^{\frac{1}{4-q\delta_q}}$. 
Then (\ref{eqa5.5}) implies that
\begin{equation*}
0>m(c, \mu) \geq \frac{a\mathcal{S}\Lambda}{3}\!+\!\frac{b\mathcal{S}^2{\Lambda}^2}{12}+g(||\nabla \tilde{u}_{\mu}||_{2}) \geq \frac{a\mathcal{S}\Lambda}{3}\!+\!\frac{b\mathcal{S}^2{\Lambda}^2}{12}+\min _{t\geq0} g(t)>0.
\end{equation*}
Consequently, up to a subsequence $u_{n} \rightarrow \tilde{u}_{\mu}$ strongly in $H^1$, $\tilde{u}_{\mu}$ is an interior local minimizer for $\left.E_{\mu}\right|_{A_{R_{0}}}$, and solves $(1.1)_{\tilde{\lambda}}$ for some $\tilde{\lambda}<0$. Moreover, $\tilde{u}_{\mu}$ is nonnegative and  radially decreasing and the strong maximum principle implies that $\tilde{u}_{\mu}>0$. Since any critical point of
$E_{\mu}|_{S_{c}}$ lies in $\mathcal{P}_{c,\mu}$ and $m(c, \mu)=\inf _{\mathcal{P}_{c,\mu}} E_{\mu}$ (see Lemma \ref{lem3.6}), we see that $\tilde{u}_{\mu}$ is a ground state for $\left.E_{\mu}\right|_{S_c}$. Similar to the proof of Theorem \ref{th1.1}-(1), we can show that any ground state of $E_{\mu}|_{S_c}$ is a local minimizer of $E_{\mu}$ in $A_{R_0}$. \qed \\ 

To obtain the asymptotic property of $m(c, \mu)$ and $\sigma(c, \mu)$ as $\mu \to 0^{+}$, we need to study equation $(1.1)_{\lambda}$ with $\mu=0$. Although it has been studied in \cite{y,XzYz}, we still give a detailed proof as we obtain a ground state solution. Modify the arguments in Section 2, especially Lemma \ref{lem3.3} and Lemma \ref{lem3.4}, we can derive the following Lemmas \ref{lem6.1}-\ref{lem6.2}.

\begin{lemma}\label{lem6.1}
Let $a\!>\!0$, $b\!>\!0$, $c\!>\!0$, $\frac{14}{3}\!<\!p\!<\!6$ and $\mu\!=\!0$. Then $\mathcal{P}_{0}^{c, \mu}=\emptyset$, and $\mathcal{P}_{c, \mu}$ is a smooth manifold of codimension 2 in $H^{1}(\R^3)$.
\end{lemma}
\begin{proof}
The proof is similar to that of Lemma \ref{lem3.3}.
\end{proof}

\begin{lemma}\label{lem6.2}
Let $a\!>\!0$, $b\!>\!0$, $c\!>\!0$, $\frac{14}{3}\!<\!p\!<\!6$ and $\mu\!=\!0$. For every $u \in S_c$, there exists a unique $t_{u} \in \mathbb{R}$ such that $t_{u} \star u \in \mathcal{P}_{c,\mu}$. $t_{u}$ is the unique critical point of the function $\Psi_{u}^{\mu}$, and is a strict maximum point at positive level. Moreover: \\
\noindent $(1)$ $\mathcal{P}_{c,\mu}=\mathcal{P}_{-}^{c,\mu}$.\\
\noindent $(2)$ $\Psi_{u}^{\mu}$ is strictly decreasing and concave on $\left(t_{u},+\infty\right)$.\\
$(3)$ The maps $u \in S_c \mapsto t_{u} \in \mathbb{R}$ are of class $C^1$.\\
$(4)$ If $P_{\mu}(u)<0$, then $t_{u}<0$.
\end{lemma}

\begin{proof}
The proof is similar to that of Lemma 6.1 in \cite{NSaE}. 
\end{proof}


\begin{lemma}\label{lem6.3}
Let $a\!>\!0$, $b\!>\!0$, $c\!>\!0$, $\frac{14}{3}\!<\!p\!<\!6$ and $\mu\!=\!0$, then
$m(c, 0) \!:=\!\inf _{u \in \mathcal{P}_{c,0}} E_{0}(u)\!>\!0$.
\end{lemma}
\begin{proof}
By (\ref{equ2.2}) and $P_{0}(u)\!=\!0$, we have $a||\nabla u||_{2}^2+b||\nabla u||_{2}^4\!=\!\delta_p{||u||}_p^p \!\leq\! \delta_p \mathcal{C}_p^{p}\left\|\nabla u\right\|_{2}^{p\delta_p} c^{p(1-\delta_p)}$. So we get $\inf _{u \in \mathcal{P}_{c,0}}||\nabla u||_{2}\geq C>0$ from $p\delta_p>4$. As $P_{0}(u)=0$, we have
$$ \inf _{u \in \mathcal{P}_{c,0}} E_{0}\left(u\right)= \inf_{u \in \mathcal{P}_{c,0}} \Big\{(\frac{a}{2}-\frac{a}{p\delta_p})||\nabla u||_{2}^{2}+(\frac{b}{4}-\frac{b}{p\delta_p})||\nabla u||_{2}^{4}\Big \}\geq C>0.$$
\end{proof}

\begin{lemma}\label{lem6.4}
Let $a\!>\!0$, $b\!>\!0$, $c\!>\!0$, $\frac{14}{3}\!<\!p\!<\!6$ and $\mu\!=\!0$. There exists $k\!>\!0$ sufficiently small such that
$$
0<\sup_{\overline{A_{k}}} E_{0}<m(c, 0) \quad \text{and} \quad u\in \overline{A_{k}} \Longrightarrow E_{0}(u)>0,~~~~P_{0}(u)>0,
$$
where $A_{k} :=\left\{u \in S_c :||\nabla u||_{2} < k\right\}$.
\end{lemma}
\begin{proof}
By using (\ref{equ2.2}), we have
\begin{align*}
E_{0}(u) \!\geq\! \frac{b{||\nabla u||}_2^4}{4}\!-\!\frac{\mathcal{C}_p^p c^{p(1-\delta_p)}}{p} {||\nabla u||}_2^{p\delta_p},~~~~
P_{0}(u)\!\geq\!  b{||\nabla u||}_2^4\!-\!\delta_p \mathcal{C}_p^p {||\nabla u||}_2^{p\delta_p} c^{p(1-\delta_p)}.
\end{align*}
Therefore, for any $u \in \overline{A_{k}}$ with $k$ small enough, we have
$$
0<\sup _{\overline{A_{k}}} E_{0} \quad \text { and }u\in \overline{A_{k}} \Longrightarrow E_{0}(u)>0,~~~~P_{0}(u)>0.
$$
If necessary replacing $k$ with a smaller quantity, we also have
$$E_{0}(u) \! \leq \! \frac{a}{2}{||\nabla u||}_2^2\!+\!\frac{b}{4}{||\nabla u||}_2^4<m(c,0),~~~~~~~~\forall u \in \overline{A_{k}}$$
since $m(c,0)>0$ by Lemma \ref{lem6.3}.
\end{proof}

\begin{lemma} \label{lemma6.5}
Let $a\!>\!0$, $b\!>\!0$, $c\!>\!0$, $\frac{14}{3}\!<\!p\!<\!6$ and $\mu\!=\!0$. Then, there exists a positive radial critical point $u_0$ for $E_0|_{S_c}$ at a positive level
$$m_r(c, 0)={m}(c,0):=\inf _{\mathcal{P}_{c,0}} E_0=E_0(u_0) $$
and as a result $u_0$ is the unique ground state of $E_0|_{S_c}$.
\end{lemma}
\begin{proof}
Utilising Lemmas \ref{lem6.1}-\ref{lem6.4} and by using the same arguments in Section 7 in \cite{NsAe}, we can drive that there exists a positive radial critical point $u_0$ for $E_0|_{S_c}$ at a Mountain Pass level $\sigma(c,0)>0$ characterized by $\sigma(c,0)=\inf _{\mathcal{P}_{c,0}\cap S_{c,r}} E_0$. By rearrangement technique and Lemma \ref{lem6.2}, we have $m_r(c, 0):=\inf _{\mathcal{P}_{c,0}\cap S_{c,r}} E_0=\inf _{\mathcal{P}_{c,0}} E_0$. Following \cite{LsCx,XzYz}, $u_0$ is unique since $u_0>0$.    
\end{proof}

\begin{lemma}\label{lem6.7}
Let $a\!>\!0$, $b\!>\!0$, $c\!>\!0$, $2\!<\!q\!<\!\frac{10}{3}$, $\frac{14}{3}\!<\!p\!<\!6$ and $0\!<\!\mu\!<\!\min\{\mu_{*},\mu^{*}\}$, then
$$
\inf _{{\mathcal{P}_{-}^{c, \mu}} \cap S_{c, r}} E_{\mu}=\inf _{u \in S_{c, r}} \max _{s \in \mathbb{R}} E_{\mu}(s \star u), \quad \text { and } \quad \inf _{{\mathcal{P}_{-}^{c, 0}} \cap S_{c, r}} E_0=\inf _{u \in S_{c, r}} \max _{s \in \mathbb{R}} E_{0}(s \star u),
$$
where $\mu_{*},\mu^{*}$ were defined in \eqref{ConAt1.10}.
\end{lemma}

\begin{proof}
$\forall u \!\in\! S_{c, r}$, by Lemma \ref{lem3.4}, there exists a unique $t_{u} \!\in \!\R$ such that $t_{u} \!\star\! u \!\in\!\mathcal{P}^{c,\mu}_{-}\!\cap\! S_{c, r}$. Thus, for any $u \!\in\! {\mathcal{P}_{-}^{c, \mu}} \!\cap \!S_{c, r}$, we have $t_{u}\!=\!0$ and
$$
E_{\mu}(u)=\max _{s \in \mathbb{R}} E_{\mu}(s \star u) \geq \inf _{v \in S_{c, r}} \max _{s \in \mathbb{R}} E_{\mu}(s \star v).
$$
On the other hand, if $u \in S_{c, r}$, then $t_{u} \!\star\! u \!\in\!\mathcal{P}^{c,\mu}_{-}\!\cap\! S_{c, r}$, and hence
$$
\max _{s \in \mathbb{R}} E_{\mu}(s \star u)=E_{\mu}\left(t_{u} \star u\right) \geq \inf _{{\mathcal{P}_{-}^{c, \mu}} \!\cap \!S_{c, r}} E_{\mu}.
$$
%
By using Lemma \ref{lem6.2}, we can similarly prove
$$\inf _{{\mathcal{P}_{-}^{c, 0}} \cap S_{c, r}} E_0\!=\!\inf _{u \in S_{c, r}} \max _{s \in \mathbb{R}} E_{0}(s \star u).$$
\end{proof}

\begin{lemma}\label{lem6.8}
Let $a\!>\!0$, $b\!>\!0$, $c\!>\!0$, $2\!<\!q\!<\!\frac{10}{3}$ and  $\frac{14}{3}\!<\!p\!<\!6$. For any $0\!\leq\!\mu_{1}\!<\!\mu_2\!<\!\min\{\mu_{*},\mu^{*}\}$, it holds that
$\sigma\left(c, \mu_{2}\right) \leq \sigma\left(c, \mu_{1}\right) \leq m(c, 0)$, where $\mu_{*},\mu^{*}$ were defined in \eqref{ConAt1.10}.
\end{lemma}

\begin{proof}
From (\ref{eq5.3}), we have $\sigma(c, \mu)\!=\!\inf _{{\mathcal{P}_{-}^{c, \mu}} \!\cap \!S_{c, r}} E_{\mu}$. By Lemmas \ref{lemma6.5}-\ref{lem6.7}, we have
\begin{align*}
\sigma\left(c, \mu_{1}\right)=\inf _{u \in S_{c, r}} \max _{s \in \mathbb{R}} E_{\mu_{1}}\left(s \star u\right) \leq \inf _{u \in S_{c, r}} \max _{s \in \mathbb{R}} E_{0}\left(s \star u\right)=m_r(c,0)=m(c,0),\\
\sigma\left(c, \mu_{2}\right) \leq \max _{s \in \mathbb{R}} E_{\mu_{2}}\left(s \star \hat{u}_{\mu_{1}}\right) \leq \max _{s \in \mathbb{R}} E_{\mu_{1}}\left(s \star \hat{u}_{\mu_{1}}\right)=E_{\mu_{1}}\left(\hat{u}_{\mu_{1}}\right)=\sigma\left(c, \mu_{1}\right).
\end{align*}
\end{proof}

\noindent \textbf{Proof of Theorem \ref{th1.1}-(4): convergence of $\tilde{u}_{\mu}$.}  \\
From Lemma \ref{lem3.1}, we know that $R_{0}(c, \mu) \rightarrow 0$ as $\mu \rightarrow 0^{+}$, and hence $||\nabla \tilde{u}_{\mu}||_{2}<R_{0}(c, \mu) \rightarrow 0$ as well. Moreover
\begin{align*}
0>m(c, \mu) \!\geq\! \frac{a}{2}{\| \nabla \tilde{u}_{\mu} \|}_2^2\!+\!\frac{b}{4}{\| \nabla \tilde{u}_{\mu} \|}_2^4\!-\! \frac{ \mathcal{C}_p^{p}}{p}\left\|\nabla \tilde{u}_{\mu}\right\|_{2}^{p\delta_p} c^{p(1-\delta_p)}\!-\!\frac{ \mu \mathcal{C}_q^{q}}{q}\left\|\nabla \tilde{u}_{\mu} \right\|_{2}^{q\delta_q} c^{q(1-\delta_q)} \rightarrow 0,
\end{align*}
which implies that $m(c,\mu)\to 0$. \qed \\

We consider now the behavior of $\hat{u}_{\mu}$.

\noindent \textbf{Proof of Theorem \ref{th1.1}-(5): convergence of ${\hat{u}_{\mu}}$.}  \\
Let us consider $\left\{\hat{u}_{\mu} \!: 0\!<\!\mu\!<\!\overline{\mu}\right\}$, with $\overline{\mu}$ small enough. Since $\hat{u}_{\mu} \!\in\! \mathcal{P}_{c, \mu}$, from Lemma \ref{lem6.8}, we have
\begin{align*}
{m}(c,0) & \geq \sigma\left(c, \mu\right) =E_{\mu}\left({\hat{u}_{\mu}} \right)=(\frac{a}{2}-\frac{a}{p\delta_p})||\nabla {\hat{u}_{\mu}}||_{2}^{2}+(\frac{b}{4}-\frac{b}{p\delta_p})||\nabla {\hat{u}_{\mu}}||_{2}^{4}-\frac{\mu}{q}
\left(1-\frac{q\delta_q}{p\delta_p}\right)\|{\hat{u}_{\mu}}\|_{q}^q \\
& \geq (\frac{a}{2}-\frac{a}{p\delta_p})||\nabla {\hat{u}_{\mu}}||_{2}^{2}+(\frac{b}{4}-\frac{b}{p\delta_p})||\nabla {\hat{u}_{\mu}}||_{2}^{4}-\frac{\mu}{q}
\left(1-\frac{q\delta_q}{p\delta_{p}}\right)
\mathcal{C}_q^{q}c^{q(1-\delta_q)} \left\|\nabla {\hat{u}_{\mu}}\right\|_{2}^{q\delta_q}.
\end{align*}
Hence $\left\{\hat{u}_{\mu}\right\}$ is bounded in $H^1$. Since each $\hat{u}_{\mu}$ is a positive function in $S_{c,r}$, we deduce that up to a subsequence $\hat{u}_{\mu} \rightharpoonup \hat{u}\geq0$ weakly in $H^1(\R^3)$, strongly in $L^{r}$ for $2<r<6$ and a.e. on $\mathbb{R}^{3}$, as $\mu \rightarrow 0^{+}$. Using the fact that $\hat{u}_{\mu}$ solves
\begin{equation}\label{eq6.2}
- \Bigl(a+b  {{{\| {\nabla \hat{u}_{\mu} } \|}_2^2}} \Bigr)\Delta \hat{u}_{\mu}
   =\hat{\lambda}_{\mu} \hat{u}_{\mu}+ {| \hat{u}_{\mu} |^{p - 2}}\hat{u}_{\mu}+\mu {| \hat{u}_{\mu} |^{q - 2}}\hat{u}_{\mu} \text { in } \mathbb{R}^{3}
\end{equation}
for $\hat{\lambda}_{\mu}\!<\!0$ and $P_{\mu}\left(\hat{u}_{\mu}\right)\!=\!0$, we infer that $\hat{\lambda}_{\mu} c^{2}\!=\!\mu(\delta_q-1)||\hat{u}_{\mu}||_{q}^{q}\!+\!(\delta_p-1) ||\hat{u}_{\mu}||_{p}^{p}$. As $\mu\!>\!0$ and $0\!<\!\delta_{q},\delta_p\!<\!1$, we deduce that $\hat{\lambda}_{\mu}$ converges (up to a subsequence) to some $\hat{\lambda} \!\leq\! 0$ satisfying
$$\hat{\lambda} c^{2}= (\delta_p-1) ||\hat{u} ||_{p}^{p},$$
with $\hat{\lambda}=0$ if and only if $\hat{u} \equiv 0$. We claim that $\hat{\lambda}<0$. In fact, $\hat{u}_{\mu} \rightharpoonup \hat{u}$ weakly in $H^1$ implies that $\hat{u}$ is a weak radial solution to
\begin{equation}\label{eq6.3}
-\Bigl( a+bB \Bigr)  \Delta \hat{u}=\hat{\lambda} \hat{u}+|\hat{u}|^{p-2} \hat{u}~~~~\mbox{in}~~~~{\R}^3,
\end{equation}
where $B\!:=\!\mathop {\lim }\limits_{\mu  \to 0^{+}}{{{\| {\nabla \hat{u}_{\mu} } \|}_2^2}}\geq{{\| {\nabla \hat{u}  } \|}_2^2}$. 
By Lemma \ref{lem6.8}, we have
\begin{align*}
-\frac{b}{4}||\nabla {\hat{u}}||_{2}^{4}+\left(\frac{\delta_p}{2}-\frac{1}{p}\right)\|{\hat{u}}\|_p^p &\geq \lim _{\mu \rightarrow 0^{+}}\left[-\frac{b}{4}||\nabla {\hat{u}_{\mu}}||_{2}^{4}+
\left(\frac{\delta_p}{2}-\frac{1}{p}\right)\|{\hat{u}_{\mu}}\|_p^p
-\mu
\left(\frac{1}{q}-\frac{\delta_q}{2}\right)\|{\hat{u}_{\mu}}\|_{q}^q \right] \\
&=\lim _{\mu \rightarrow 0^{+}} E_{\mu}\left(\hat{u}_{\mu}\right)=\lim _{\mu \rightarrow 0^{+}} \sigma(c, \mu) \geq \sigma(c, \overline{\mu})>0,
\end{align*} 
which gives $\left(\frac{\delta_p}{2}-\frac{1}{p}\right)\|{\hat{u}}\|_p^p\!>\!\frac{b}{4}||\nabla {\hat{u}}||_{2}^{4}$. So we have $\hat{u} \!\not \equiv\! 0$, and in turn yields $\hat{\lambda}\!<\!0$ and $B\!>\!0$. The strong maximum principle implies that $\hat{u}\!>\!0$. Test (\ref{eq6.2})-(\ref{eq6.3}) with $\hat{u}_{\mu}\!-\!\hat{u}$, we have
\begin{align*}
( a+bB )||\nabla \left(\hat{u}_{\mu}-\hat{u}\right)||_2^{2}-\hat{\lambda}
|| \hat{u}_{\mu}-\hat{u} ||_2^{2}\to0,
\end{align*}
which implies that $\hat{u}_{\mu} \rightarrow \hat{u}$ in $H^1$ as $\mu \rightarrow 0^{+}$. It results to $m(c, 0)\!\leq\!E_{0}(\hat{u})$. Since $\mathop {\lim }\limits_{\mu  \to 0^{+}}{{{\| {\nabla \hat{u}_{\mu} } \|}_2^2}}={{\| {\nabla \hat{u}  } \|}_2^2}$, we also have
\begin{align*}
E_{0}(\hat{u})=\frac{a}{2}||\nabla {\hat{u}}||_{2}^{2}+\frac{b}{4}||\nabla {\hat{u}}||_{2}^{4}-\frac{1}{p}\|{\hat{u}}\|_p^p
\!=\!\lim _{\mu \rightarrow 0^{+}} E_{\mu}\left(\hat{u}_{\mu}\right)\!=\!\lim _{\mu \rightarrow 0^{+}} \sigma(c, \mu) \!\leq\! m(c, 0).
\end{align*}
Consequently, $E_{0}(\hat{u})\!=\!\mathop {\lim }\limits_{\mu  \to 0^{+}}\sigma(c, \mu)\!=\!{m}(c, 0)$ and $\hat{u}$ is a positive solution to (\ref{eq6.3}). From \cite{Kwon,LsCx,XzYz}, we know that (\ref{eq6.3}) has a unique positive solution $u_0$. Thus $\hat{u}=u_0$.  \qed \\

\noindent \textbf{Proof of Theorem \ref{th1.3}-(3):} \\
From Lemma \ref{lem3.1}, we know that $R_{0}(c, \mu) \rightarrow 0$ as $\mu \rightarrow 0^{+}$, and hence $||\nabla \tilde{u}_{\mu}||_{2}<R_{0}(c, \mu) \rightarrow 0$ as well. Moreover
\begin{align*}
0>m(c, \mu)=E_{\mu}\left(\tilde{u}_{\mu}\right)
\!\geq\! \frac{a}{2}{\| \nabla \tilde{u}_{\mu} \|}_2^2\!+\!\frac{b}{4}{\| \nabla \tilde{u}_{\mu} \|}_2^4\!-\! \frac{ \mathcal{S}^{-3}}{6} \left\|\nabla \tilde{u}_{\mu}\right\|_{2}^6\!-\!\frac{ \mu \mathcal{C}_q^{q}}{q}\left\|\nabla \tilde{u}_{\mu} \right\|_{2}^{q\delta_q} c^{q(1-\delta_q)} \rightarrow 0,
\end{align*}
which implies that $m(c,\mu)\to 0$. \qed \\

\section{Purely $L^2$-supercritical case}
In this Section, we always assume that $\frac{14}{3}\!<\!q\!<\!p\!\leq\!6$. Under this setting, we obtain one critical point for $E_{\mu}|_{S_c}$, since $E_{\mu}|_{S_c}$ admits a Mountain Pass geometry. Subsection 5.1 is devoted to locating the exact position of some critical points to $E_{\mu}|_{S_c}$. In Subsection 5.2, we prove Theorems \ref{tTH1.3}-\ref{tTH1.4}.

\subsection{The exact location of some critical points to $E_{\mu}|_{S_c}$ for $\frac{14}{3}\!<\!q\!<\!p\!\leq\!6$}
In this Subsection, we study the structure of $\mathcal{P}_{c, \mu}$ and $E_{\mu}$ to locate the position of some critical points to $E_{\mu}|_{S_c}$. Since $\frac{14}{3}<q<p\leq6$, we have $4\!<\!q\delta_q\!<\!p\delta_p$. Similar to the proof of Lemmas \ref{lem3.3}-\ref{lemma1.3}, we can prove that $\mathcal{P}_{c, \mu}$ is a natural constraint and $\mathcal{P}_{0}^{c, \mu}=\emptyset$. Furthermore, we have
%
%

\begin{lemma}  \label{LeMal2.8}
Let $\tilde{a},\tilde{b},\tilde{c},\tilde{d},\tilde{p},\tilde{q} \in(0,+\infty)$ and   $f(t):=\tilde{a}t^2+\tilde{b}t^4-\tilde{c}t^{\tilde{p}}-\tilde{d}t^{\tilde{q}}$ for $t\geq0$. If $\tilde{p},\tilde{q} \in(4,+\infty)$, $f(t)$ has a unique maximum point at a positive level on $[0,+\infty)$.
\end{lemma}
\begin{proof}
Direct calculations give
\begin{align*}
&f'(t)=tg(t) ~~~~\mbox{for}~~~~ g(t)=2\tilde{a}+4\tilde{b}t^2-{\tilde{p}}\tilde{c}t^{\tilde{p}-2}
-{\tilde{q}}\tilde{d}t^{\tilde{q}-2};  \\
&g'(t)=tw(t) ~~~~\mbox{for}~~~~ w(t)=8\tilde{b}-{\tilde{p}}(\tilde{p}-2)\tilde{c}t^{\tilde{p}-4}
-{\tilde{q}}(\tilde{q}-2)\tilde{d}t^{\tilde{q}-4};\\
&w'(t)=-{\tilde{p}}(\tilde{p}-2)(\tilde{p}-4)\tilde{c}t^{\tilde{p}-5}
-{\tilde{q}}(\tilde{q}-2)(\tilde{q}-4)\tilde{d}t^{\tilde{q}-5}.
\end{align*}
Since $w'(t)<0$ for $t>0$, we know that $w(t) \searrow$ on $[0,+\infty)$. The fact that  $w(0)>0$ and $w(+\infty)=-\infty$ imply that there exists unique $t^*>0$ such that $w(t^*)=0$, $w(t)>0$ if $t\in(0,t^*)$ and $w(t)<0$ if $t\in(t^*,+\infty)$. Consequently, $g(t)\nearrow$ on $[0,t^*)$ and $ \searrow$ on $(t^*,+\infty)$. The fact that  $g(0)>0$ and $g(+\infty)=-\infty$ imply that there exists unique $\bar{t}>t^*$ such that $g(\bar{t})=0$, $g(t)>0$ if $t\in(0,\bar{t})$ and $g(t)<0$ if $t\in(\bar{t},+\infty)$. We get $f'(t)>0$ if $t\in(0,\bar{t})$ and $f'(t)<0$ if $t\in(\bar{t},+\infty)$, which implies that $f(t)\nearrow$ on $[0,\bar{t})$ and $\searrow$ on $(\bar{t},+\infty)$. Since $f(0)=0$, then $f(t)$ has a unique maximum point at $\bar{t}$ and $f(\bar{t})>0$.
\end{proof}

\begin{lemma}\label{LeM3.4}
Let $a\!>\!0$, $b\!>\!0$, $c\!>\!0$, $\frac{14}{3}\!<\!q\!<\!p\!\leq\!6$ and $\mu\!>\!0$. For every $u \!\in\! S_{c}$, $\Psi_{u}^{\mu}$ has a unique critical point $t_{u} \!\in\! \mathbb{R}$, which is a strict maximum point at a positive level. Moreover:\\
\noindent $(1)$ $\mathcal{P}_{c,\mu}=\mathcal{P}_{-}^{c,\mu}$.\\
\noindent $(2)$ $\Psi_{u}^{\mu}$ is strictly decreasing on $\left(t_{u},+\infty\right)$, and $t_u<0$ implies $P_{\mu}(u)<0$.\\
$(3)$ The maps $u \in S_c \mapsto t_{u} \in \mathbb{R}$ are of class $C^1$.\\
$(4)$ If $P_{\mu}(u)<0$, then $t_{u}<0$.
\end{lemma}

\begin{proof}
By using Lemma \ref{LeMal2.8}, we derive that $\Psi_{u}^{\mu}$ has a unique maximum point at a positive level. The rest of the proof is similar to that of Lemma 6.1 in \cite{NSaE}.
\end{proof}

\begin{lemma}\label{lEM7.3}
Let $a\!>\!0$, $b\!>\!0$, $c\!>\!0$, $\frac{14}{3}\!<\!q\!<\!p\!\leq\!6$ and $\mu\!>\!0$. Then, we have $$ m(c,\mu) :=\inf _{u \in \mathcal{P}_{c,\mu}} E_{\mu}(u)>0.$$
\end{lemma}

\begin{proof}
The proof is similar to that of Lemma \ref{lem6.3}.
\end{proof}

\begin{lemma}\label{lEM7.4}
Let $a\!>\!0$, $b\!>\!0$, $c\!>\!0$, $\frac{14}{3}\!<\!q\!<\!p\!\leq\!6$ and $\mu\!>\!0$. Then, there exists $k>0$ sufficiently small such that
$$
0<\sup_{\overline{A_{k}}} E_{\mu}<m(c,\mu) \quad \text{and} \quad u\in \overline{A_{k}} \Longrightarrow E_{\mu}(u)>0,~~~~P_{\mu}(u)>0,$$
where $A_{k} :=\left\{u \in S_c :||\nabla u||_{2}^{2} < k\right\}$.
\end{lemma}

\begin{proof}
The proof is similar to that of Lemma \ref{lem6.4}.
\end{proof}

To apply Proposition \ref{prp4.2} and recover compactness when $p=6$, we need an estimate from above on
 $$m_{r}(c, \mu):=\inf _{u \in \mathcal{P}_{c,\mu} \cap S_{c,r} } E_{\mu}(u).$$

\begin{lemma}\label{lemma7.5}
Let $a\!>\!0$, $b\!>\!0$, $c\!>\!0$, $\frac{14}{3}\!<\!q\!<\!6$, $p\!=\!6$ and $\mu\!>\!0$. Then $m_{r}(c, \mu)\!<\!\frac{a\mathcal{S}\Lambda}{3}\!+\!\frac{b\mathcal{S}^2{\Lambda}^2}{12}$, where $\Lambda=\frac{b{\mathcal{S}}^2}{2}
+\sqrt{a\mathcal{S}+\frac{b^2{\mathcal{S}}^4}{4}}$.
\end{lemma}

\begin{proof}
By Theorem 1.42 of \cite{Mwlm}, we know that $\mathcal{S}=\inf _{u \in {{D}}^{1,2}(\mathbb{R}^{3})\setminus \{0\} }    \frac{\left\|\nabla u\right\|_{2}^{2}}{||u||_{6}^{2}}$ is attained by
\begin{align}  \label{algn5.11}
U_{\varepsilon}(x):=3^{\frac{1}{4}}\left(\frac{\varepsilon}
{\varepsilon^{2}+|x|^{2}}\right)^{\frac{1}{2}},~~~~\forall\varepsilon>0.
\end{align}
Furthermore, we have $\left\|\nabla U_{\varepsilon} \right\|_{2}^{2}=||U_{\varepsilon}||_6^6=\mathcal{S}^{\frac{3}{2}}$. Take a radially decreasing cut-off function $\eta \in C_{c}^{\infty}\left(\mathbb{R}^3\right)$ such that $\eta\equiv1$ in $B_1(0)$, $\eta\equiv0$ in $B^c_2(0):=\mathbb{R}^3\setminus B_2(0)$, and let
$$
u_{\varepsilon}(x):=\eta(x) U_{\varepsilon}(x), \quad \text { and } \quad v_{\varepsilon}(x):=c \frac{u_{\varepsilon}(x)}{||u_{\varepsilon}||_{2}},~~~~ \forall\varepsilon\in(0,1).
$$
Clearly, $v_{\varepsilon}\in S_{c,r}$, by Lemma \ref{LeM3.4}, there exists a unique $t_{v_{\varepsilon},\mu}\in\R$ such that
$$
m_{r}(c, \mu)=\inf _{\mathcal{P}_{c, \mu} \cap S_{c,r}} E_{\mu} \leq E_{\mu}\left(t_{v_{\varepsilon}, \mu} \star v_{\varepsilon}\right)=\max _{s \in \mathbb{R}} E_{\mu}\left(s \star v_{\varepsilon}\right)=\max _{s \in \mathbb{R}} \Psi_{v_{\varepsilon}}^{\mu}(s), \quad \forall \varepsilon>0.
$$
So, it is sufficient to prove $ \max _{s \in \mathbb{R}} \Psi_{v_{\varepsilon}}^{\mu}(s)\!=\!E_{\mu}\left(t_{v_{\varepsilon}, \mu} \star v_{\varepsilon}\right)\!<\! \frac{a\mathcal{S}\Lambda}{3}\!+\!\frac{b\mathcal{S}^2{\Lambda}^2}{12}$.

To this end, we need some integral estimates. Similar to Lemma 1.46 in \cite{Mwlm} or Lemma A.1 in \cite{NSaE}, we can derive that
\begin{align}  \label{inTegral4.2} 
&||\nabla u_{\varepsilon}||_{2}^2 \!=\! {\mathcal{S}}
^{\frac{3}{2}}\!+\!O(\varepsilon),~~~~
||u_{\varepsilon}||_6^6\!=\!{\mathcal{S}}
^{\frac{3}{2}}\!+\!O(\varepsilon^3),
~~~~||u_{\varepsilon}||_{2}^2=O(\varepsilon),
||u_{\varepsilon}||_{q}^q=O(\varepsilon^{3-\frac{q}{2}}), \nonumber \\
&||\nabla u_{\varepsilon} ||_{2}^2 \geq C_1,~~~~~~~~\frac{1}{C_2}\geq {||u_{\varepsilon}||}_6^6\geq C_2,~~~~~~~~||u_{\varepsilon}||_{2}^2 \geq C_3\varepsilon
\end{align}
for some constants $C_i>0$ ($i=1,2,3$), which are independent of $\varepsilon$, $c$ and $\mu$.

Next, we prove $ \max _{s \in \mathbb{R}} \Psi_{v_{\varepsilon}}^{0}(s)=E_{0}\left(t_{v_{\varepsilon},0} \star v_{\varepsilon}\right)=\frac{a\mathcal{S}\Lambda}{3}\!+\!\frac{b\mathcal{S}^2{\Lambda}^2}{12}
+O(\varepsilon^{\frac{1}{2}})$. Since
$$
\Psi_{v_{\varepsilon}}^{0}(s)=\frac{a}{2}e^{2s} ||\nabla v_{\varepsilon} ||_{2}^2+\frac{b}{4}e^{4s} ||\nabla v_{\varepsilon} ||_{2}^4-\frac{e^{ 6s}}{6} {||v_{\varepsilon}||}_6^6,
$$
we see that $\Psi_{v_{\varepsilon}}^{0}(s)$ has a unique maximum point $t_{v_{\varepsilon},0}$ such that
$$e^{2t_{v_{\varepsilon},0}}\!=\!\frac{b{||\nabla v_{\varepsilon}||}_{2}^{4}}{2||v_{\varepsilon}||
_6^6}+\sqrt{\frac{a{||\nabla v_{\varepsilon}||}_{2}^{2}}{||v_{\varepsilon}||
_6^6}+\frac{b^2{||\nabla v_{\varepsilon}||}_{2}^{8}}{4||v_{\varepsilon}||
_6^{12}}}.$$
Then, we drive that
\begin{align*}
\frac{c^2e^{2t_{v_{\varepsilon},0}}}{{||  u_{\varepsilon}||}_{2}^{2}}&=\frac{b{||\nabla u_{\varepsilon}||}_{2}^{4}}{2||u_{\varepsilon}||
_6^6}+\sqrt{\frac{a{||\nabla u_{\varepsilon}||}_{2}^{2}}{||u_{\varepsilon}||
_6^6}+\frac{b^2{||\nabla u_{\varepsilon}||}_{2}^{8}}{4||u_{\varepsilon}||
_6^{12}}}\\
&=\frac{b{\big({\mathcal{S}}
^{\frac{3}{2}}\!+\!O(\varepsilon)\big)}^{2}}{2\big({\mathcal{S}}
^{\frac{3}{2}}\!+\!O(\varepsilon^3)\big)}+\sqrt{\frac{a\big({\mathcal{S}}
^{\frac{3}{2}}\!+\!O(\varepsilon)\big) }{{\mathcal{S}}
^{\frac{3}{2}}\!+\!O(\varepsilon^3)}+\frac{b^2{\big({\mathcal{S}}
^{\frac{3}{2}}\!+\!O(\varepsilon)\big)}^{4}}{4\big({\mathcal{S}}
^{\frac{3}{2}}\!+\!O(\varepsilon^3)\big)^{2}}} \\
&=\frac{b{\mathcal{S}}
^{\frac{3}{2}}}{2}+\sqrt{a+\frac{b^2{\mathcal{S}}
^{3}}{4}\!+\!O(\varepsilon) }\!+\!O(\varepsilon)\\
&\leq\frac{b{\mathcal{S}}
^{\frac{3}{2}}}{2}+\sqrt{a+\frac{b^2{\mathcal{S}}
^{3}}{4}  }\!+\!O(\varepsilon^{\frac{1}{2}})
=\frac{\Lambda}{\sqrt{\mathcal{S}}}\!+\!O(\varepsilon^{\frac{1}{2}}),
\end{align*}
where $\Lambda=\frac{b{\mathcal{S}}^2}{2}
+\sqrt{a\mathcal{S}+\frac{b^2{\mathcal{S}}^4}{4}}$. This leads to that
\begin{align}\label{tDEFI11}
\sup_{s \in \mathbb{R}} \Psi_{v_{\varepsilon}}^{0}(s)&=\Psi_{v_{\varepsilon}}^{0}(t_{v_{\varepsilon},0})
=\frac{a}{2}\frac{c^2e^{2t_{v_{\varepsilon},0}}}{{||  u_{\varepsilon}||}_{2}^{2}} ||\nabla u_{\varepsilon} ||_{2}^2+\frac{b}{4}\frac{c^4e^{4t_{v_{\varepsilon},0}}}{{||  u_{\varepsilon}||}_{2}^{4}} ||\nabla u_{\varepsilon} ||_{2}^4-\frac{c^6e^{6t_{v_{\varepsilon},0}}}{{||  u_{\varepsilon}||}_{2}^{6}}  \frac{{||u_{\varepsilon}||}_6^6}{6}
\nonumber\\
=&\frac{a}{2}\frac{c^2e^{2t_{v_{\varepsilon},0}}}{{||  u_{\varepsilon}||}_{2}^{2}} \Big({\mathcal{S}}
^{\frac{3}{2}}\!+\!O(\varepsilon)\Big)+\frac{b}{4}\frac{c^4e^{4t_{v_{\varepsilon},0}}}{{||  u_{\varepsilon}||}_{2}^{4}} \Big({\mathcal{S}}
^{\frac{3}{2}}\!+\!O(\varepsilon)\Big)^2-\frac{c^6e^{6t_{v_{\varepsilon},0}}}{{||  u_{\varepsilon}||}_{2}^{6}}  \frac{\big({\mathcal{S}}
^{\frac{3}{2}}\!+\!O(\varepsilon^3)\big)}{6}
\nonumber\\
\leq&\frac{a}{2}\Big( \frac{\Lambda}{\sqrt{\mathcal{S}}}\!+\!O(\varepsilon^{\frac{1}{2}}) \Big) \Big({\mathcal{S}}
^{\frac{3}{2}}\!+\!O(\varepsilon)\Big)+\frac{b}{4}\Big( \frac{\Lambda}{\sqrt{\mathcal{S}}}\!+\!O(\varepsilon^{\frac{1}{2}}) \Big)^2 \Big({\mathcal{S}}
^{3}\!+\!O(\varepsilon)\Big)    \nonumber\\
&-\Big( \frac{b{\mathcal{S}}
^{\frac{3}{2}}}{2}+\sqrt{a+\frac{b^2{\mathcal{S}}
^{3}}{4}\!+\!O(\varepsilon) }\!+\!O(\varepsilon) \Big)^3  \frac{\big({\mathcal{S}}
^{\frac{3}{2}}\!+\!O(\varepsilon^3)\big)}{6}
\nonumber\\
\leq&\frac{a\Lambda\mathcal{S}}{2}
+\frac{b\Lambda^2\mathcal{S}^2}{4}\!+\!O(\varepsilon^{\frac{1}{2}}) -\Big( \frac{b{\mathcal{S}}
^{\frac{3}{2}}}{2}+\sqrt{a+\frac{b^2{\mathcal{S}}
^{3}}{4}  }  \Big)^3  \frac{ {\mathcal{S}}
^{\frac{3}{2}} }{6}
\nonumber\\
=&\frac{a\Lambda\mathcal{S}}{2}
+\frac{b\Lambda^2\mathcal{S}^2}{4}-\frac{\Lambda^3}{6}
\!+\!O(\varepsilon^{\frac{1}{2}})=\frac{a\mathcal{S}\Lambda}{3}\!
+\!\frac{b\mathcal{S}^2{\Lambda}^2}{12}
\!+\!O(\varepsilon^{\frac{1}{2}}).
\end{align}

Finally, we estimate $t_{v_{\varepsilon},\mu}$. From $(\Psi_{v_{\varepsilon}}^{\mu})'(t_{v_{\varepsilon},\mu})
=P_{\mu}(t_{v_{\varepsilon},\mu}\star v_{\varepsilon})=0$, we have
$$ ae^{2t_{v_{\varepsilon},\mu}} ||\nabla v_{\varepsilon} ||_{2}^2+be^{4t_{v_{\varepsilon},\mu}} ||\nabla v_{\varepsilon} ||_{2}^4=\mu \delta_{q} e^{ q\delta_{q} t_{v_{\varepsilon},\mu} } {||v_{\varepsilon}||}_q^q+e^{ 6t_{v_{\varepsilon},\mu}} {||v_{\varepsilon}||}_6^6\geq e^{ 6t_{v_{\varepsilon},\mu}} {||v_{\varepsilon}||}_6^6.$$
It results to that $e^{ 2 t_{v_{\varepsilon},\mu} }\leq e^{ 2 t_{v_{\varepsilon},0} }$, so we have
\begin{align} \label{lign7.1}
  e^{  2t_{v_{\varepsilon},\mu} }\leq\frac{b{||\nabla v_{\varepsilon}||}_{2}^{4}}{2||v_{\varepsilon}||
_6^6}+\sqrt{\frac{a{||\nabla v_{\varepsilon}||}_{2}^{2}}{||v_{\varepsilon}||
_6^6}+\frac{b^2{||\nabla v_{\varepsilon}||}_{2}^{8}}{4||v_{\varepsilon}||
_6^{12}}}\leq \frac{b{||\nabla v_{\varepsilon}||}_{2}^{4}}{||v_{\varepsilon}||
_6^6}+\frac{\sqrt{a}{||\nabla v_{\varepsilon}||}_{2} }{||v_{\varepsilon}||
_6^3}.
\end{align}
On the other hand, we have
\begin{align*}
 e^{ 4t_{v_{\varepsilon},\mu}}\!=\!\frac{ a  ||\nabla v_{\varepsilon} ||_{2}^2}{
 {||v_{\varepsilon}||}_6^6 }\!+\!\frac{ b  ||\nabla v_{\varepsilon} ||_{2}^4}{
 {||v_{\varepsilon}||}_6^6 }e^{ 2 t_{v_{\varepsilon},\mu} }
\!-\!\mu\delta_q\frac{ {||v_{\varepsilon}||}_q
^q }
{{||v_{\varepsilon}||}_6
^6}e^{(q\delta_q-2) t_{v_{\varepsilon},\mu} }\!\geq\! \frac{ b  ||\nabla v_{\varepsilon} ||_{2}^4}{
 {||v_{\varepsilon}||}_6^6 }e^{ 2 t_{v_{\varepsilon},\mu} }
\!-\!\mu\delta_q\frac{ {||v_{\varepsilon}||}_q
^q }
{{||v_{\varepsilon}||}_6
^6}e^{(q\delta_q-2) t_{v_{\varepsilon},\mu} }.
\end{align*} 
By the inequality $(\ell_1+\ell_2)^{\frac{q\delta_q-4}{2}}\leq \ell_1^{\frac{q\delta_q-4}{2}}+\ell_2^{\frac{q\delta_q-4}{2}}$ for $\ell_1,\ell_2\geq0$ and (\ref{lign7.1}), we have
\begin{align*}
 e^{ 2t_{v_{\varepsilon},\mu}}&\!\geq\! \frac{ b  ||\nabla v_{\varepsilon} ||_{2}^4}{
 {||v_{\varepsilon}||}_6^6 }
\!-\!\mu\delta_q\frac{ {||v_{\varepsilon}||}_q
^q }{{||v_{\varepsilon}||}_6
^6}e^{(q\delta_q-4) t_{v_{\varepsilon},\mu} } =\frac{b {||u_{\varepsilon}||}_{2}^{2}}{c^{2}}
\frac{||\nabla u_{\varepsilon} ||_{2}^4 } {{||u_{\varepsilon}||}_6^6}-\mu\delta_q \frac{{||u_{\varepsilon}||}_{2}^{6-q}}{c^{6-q}} \frac{{||u_{\varepsilon}||}_q
^q}{{||u_{\varepsilon}||}_6
^6}e^{(q\delta_q-4) t_{v_{\varepsilon},\mu} }\\
&\geq\frac{b {||u_{\varepsilon}||}_{2}^{2}}{c^{2}}
\frac{||\nabla u_{\varepsilon} ||_{2}^4 } {{||u_{\varepsilon}||}_6^6}-\mu\delta_q \frac{{||u_{\varepsilon}||}_{2}^{6-q}}{c^{6-q}} \frac{{||u_{\varepsilon}||}_q
^q}{{||u_{\varepsilon}||}_6
^6}\Big[\frac{b{||\nabla v_{\varepsilon}||}_{2}^{4}}{||v_{\varepsilon}||
_6^6}+\frac{\sqrt{a}{||\nabla v_{\varepsilon}||}_{2} }{||v_{\varepsilon}||
_6^3}\Big]^{ \frac{q\delta_q-4}{2} }
\\
&\geq\frac{b {||u_{\varepsilon}||}_{2}^{2}}{c^{2}}
\frac{||\nabla u_{\varepsilon} ||_{2}^4 } {{||u_{\varepsilon}||}_6^6}-\mu\delta_q   \frac{{||u_{\varepsilon}||}_{2}^{6-q}}{c^{6-q}} \frac{{||u_{\varepsilon}||}_q
^q}{{||u_{\varepsilon}||}_6
^6}\Big[\Big(\frac{b{||\nabla v_{\varepsilon}||}_{2}^{4}}{||v_{\varepsilon}||
_6^6}\Big)^{ \frac{q\delta_q-4}{2} }+\Big(\frac{\sqrt{a}{||\nabla v_{\varepsilon}||}_{2} }{||v_{\varepsilon}||
_6^3}\Big)^{ \frac{q\delta_q-4}{2} }\Big]
\\
&=\frac{b {||u_{\varepsilon}||}_{2}^{2}}{c^{2}}
\frac{||\nabla u_{\varepsilon} ||_{2}^4 } {{||u_{\varepsilon}||}_6^6}-\mu\delta_q  \frac{{||u_{\varepsilon}||}_{2}^{2-q(1-\delta_q)}}{c^{2-q(1-\delta_q)}} \frac{{||u_{\varepsilon}||}_q
^q}{{||u_{\varepsilon}||}_6
^6}\Big[\Big(\frac{b{||\nabla u_{\varepsilon}||}_{2}^{4}}{||u_{\varepsilon}||
_6^6}\Big)^{ \frac{q\delta_q-4}{2} }+\Big(\frac{\sqrt{a}{||\nabla u_{\varepsilon}||}_{2} }{||u_{\varepsilon}||
_6^3}\Big)^{ \frac{q\delta_q-4}{2} }\Big]
\\
&=\frac{{||u_{\varepsilon}||}_{2}^{2}}{c^{2}} \Big\{
\frac{b ||\nabla u_{\varepsilon} ||_{2}^4 } {{||u_{\varepsilon}||}_6^6}-\frac{\mu\delta_q c^{q(1-\delta_q)} }{{||u_{\varepsilon}||}_6
^6} \frac{{||u_{\varepsilon}||}_q
^q}
{{||u_{\varepsilon}||}_{2}^{ q(1-\delta_q)}} \Big[\Big(\frac{b{||\nabla u_{\varepsilon}||}_{2}^{4}}{||u_{\varepsilon}||
_6^6}\Big)^{ \frac{q\delta_q-4}{2} }+\Big(\frac{\sqrt{a}{||\nabla u_{\varepsilon}||}_{2} }{||u_{\varepsilon}||
_6^3}\Big)^{ \frac{q\delta_q-4}{2} }\Big] \Big\}
\\
&\geq\frac{{||u_{\varepsilon}||}_{2}^{2}}{c^{2}} \Big\{
C_4-\mu\delta_q c^{q(1-\delta_q)}C_5 \frac{{||u_{\varepsilon}||}_q
^q}
{{||u_{\varepsilon}||}_{2}^{ q(1-\delta_q)}}   \Big\},
\end{align*}
where $C_4=C_4(b,\mathcal{S})>0$ and $C_5=C_5(a,b,q,\mathcal{S})>0$. Utilizing \eqref{inTegral4.2}, we have $\frac{ {||u_{\varepsilon}||}_q^q}{{||u_{\varepsilon}||}_{2}^{q(1-\delta_q)}    }=O(\varepsilon^{\frac{6-q}{4}})$. Consequently, we get
\begin{align} \label{ebounD1}
 e^{ 2t_{v_{\varepsilon},\mu}} \geq\frac{{||u_{\varepsilon}||}_{2}^{2}}{c^{2}} \Big\{
C_4-O(\varepsilon^{\frac{6-q}{4}})  \mu\delta_q c^{q(1-\delta_q)}C_5   \Big\}\geq\frac{{||u_{\varepsilon}||}_{2}^{2}}{c^{2}} \frac{C_4}{4}
\end{align}
for $\varepsilon\!>\!0$ sufficiently small. Then \eqref{ebounD1} gives
$e^{ t_{v_{\varepsilon},\mu}}\!\geq\! C \frac{{||u_{\varepsilon}||}_{2}}{c}$
for some constant $C\!=\!\frac{\sqrt{C_4}}{2}$. Since $q\!\in \! (\frac{14}{3},6)$, we get
\begin{align*}
\sup_{s \in \mathbb{R}} \Psi_{v_{\varepsilon}}^{\mu}(s)&=\Psi_{v_{\varepsilon}}^{\mu}(t_{v_{\varepsilon}, \mu})=\Psi_{v_{\varepsilon}}^{0}(t_{v_{\varepsilon}, \mu})-\mu \frac{e^{ q\delta_{q} t_{v_{\varepsilon}, \mu}}}{q} {||v_{\varepsilon}||}_q^q
\leq \sup_{s \in \mathbb{R}}\Psi_{v_{\varepsilon}}^{0}(s)-\mu \frac{e^{ q\delta_{q} t_{v_{\varepsilon}, \mu}}}{q} {||v_{\varepsilon}||}_q^q\\
&=\Psi_{v_{\varepsilon}}^{0}(t_{v_{\varepsilon}, 0})-\mu \frac{e^{ q\delta_{q} t_{v_{\varepsilon}, \mu}}}{q} {||v_{\varepsilon}||}_q^q \leq \frac{a\mathcal{S}\Lambda}{3}\!+\!\frac{b\mathcal{S}^2{\Lambda}^2}{12}
\!+\!O(\varepsilon^{\frac{1}{2}})-\frac{\mu C^{q\delta_q} c^{q(1-\delta_q)} }{q}\frac{ {||u_{\varepsilon}||}_q^q }{{||u_{\varepsilon}||}_2
^{q(1-\delta_q)}} \\
&\leq \frac{a\mathcal{S}\Lambda}{3}\!+\!\frac{b\mathcal{S}^2{\Lambda}^2}{12}
\!+\!O(\varepsilon^{\frac{1}{2}})-O(\varepsilon^{\frac{6-q}{4}})
<\frac{a\mathcal{S}\Lambda}{3}\!+\!\frac{b\mathcal{S}^2{\Lambda}^2}{12}.
\end{align*}
\end{proof}

\subsection{The existence and asymptotic results for $\frac{14}{3}\!<\!q\!<\!p\!\leq\!6$}
In this Subsection, we first prove the existence results, i.e. Theorem \ref{tTH1.3}-(1),(2) and Theorem \ref{tTH1.4}-(1),(2). Then, we prove the asymptotic results, i.e. Theorem \ref{tTH1.3}-(3) and Theorem \ref{tTH1.4}-(3).\\

To prove the asymptotic results in Theorem \ref{tTH1.4}, we need the following lemma.
\begin{lemma} \label{lemma6.5.1}
Let $a\!>\!0$, $b\!>\!0$, $c\!>\!0$, $\!p\!=\!6$ and $\mu\!=\!0$. Then,
\begin{align}  \label{algn6.4.1}
{m}_r(c,0)\!=\!{m}(c,0):=\inf _{\mathcal{P}_{c,0}} E_0=\inf _{u \in S_c} \max _{s \in \mathbb{R}} E_{0}(s \star u)=\frac{a\mathcal{S}\Lambda}{3}\!+\!\frac{b\mathcal{S}^2{\Lambda}^2}{12},
\end{align}
where $\Lambda=\frac{b{\mathcal{S}}^2}{2}
+\sqrt{a\mathcal{S}+\frac{b^2{\mathcal{S}}^4}{4}}$.
\end{lemma}

\begin{proof}
Imitate the proof of Lemma \ref{lem6.7}, we get $\inf _{\mathcal{P}_{c,0}} E_0=\inf _{u \in S_c} \max _{s \in \mathbb{R}} E_{0}(s \star u)$.
Now, we prove that $\inf _{u \in S_c} \max _{s \in \mathbb{R}} E_{0}(s \star u)=\frac{a\mathcal{S}\Lambda}{3}\!+\!\frac{b\mathcal{S}^2{\Lambda}^2}{12}$. In fact, direct calculation implies that $\max _{s \in \mathbb{R}} E_{0}(s \star u)=\Psi_{u}^{0}(t_{u,0})$ with $$e^{2t_{u,0}}\!=\!\frac{b{||\nabla u||}_{2}^{4}}{2||u||
_6^6}+\sqrt{\frac{a{||\nabla u||}_{2}^{2}}{||u||
_6^6}+\frac{b^2{||\nabla u||}_{2}^{8}}{4||u||
_6^{12}}}.$$
We claim that
\begin{align} \label{eStIofE}
\inf _{u \in S_c}e^{2t_{u,0}}{||\nabla u||}_{2}^{2}\!=\!\inf _{u \in S_c} \Big\{\frac{b{||\nabla u||}_{2}^{6}}{2||u||
_6^6}+\sqrt{\frac{a{||\nabla u||}_{2}^{6}}{||u||
_6^6}+\frac{b^2{||\nabla u||}_{2}^{12}}{4||u||
_6^{12}}}\Big\}\!=\!\mathcal{S}\Lambda.
\end{align}
On the one hand, by density of ${H}^{1}(\mathbb{R}^{3})$ in ${D}^{1,2}(\mathbb{R}^{3})$ (see \cite{NSaE}), we get
\begin{align*}
\inf _{u \in S_c}e^{2t_{u,0}}{||\nabla u||}_{2}^{2}&\!=\!\inf_{u \in {{H}}^{1}(\mathbb{R}^{3})\setminus \{0\}}e^{2t_{u,0}}{||\nabla u||}_{2}^{2}\!=\!\inf_{u \in {{D}}^{1,2}(\mathbb{R}^{3})\setminus \{0\}}e^{2t_{u,0}}{||\nabla u||}_{2}^{2}\\
&\!\geq\! \inf_{u \in {{D}}^{1,2}(\mathbb{R}^{3})\setminus \{0\}} \frac{b{||\nabla u||}_{2}^{6}}{2||u||
_6^6}+ \sqrt{\inf_{u \in {{D}}^{1,2}(\mathbb{R}^{3})\setminus \{0\}}\frac{a{||\nabla u||}_{2}^{6}}{||u||
_6^6}+\inf_{u \in {{D}}^{1,2}(\mathbb{R}^{3})\setminus \{0\}}\frac{b^2{||\nabla u||}_{2}^{12}}{4||u||
_6^{12}}}\\
&=\frac{b{\mathcal{S}}^3}{2}
+\sqrt{a\mathcal{S}^3+\frac{b^2{\mathcal{S}}^6}{4}}=\mathcal{S}\Lambda.
\end{align*}
On the other hand, since $\mathcal{S}\!=\!\inf _{u \in {{D}}^{1,2}(\mathbb{R}^{3})\!\setminus\! \{0\} }    \frac{\left\|\nabla u\right\|_{2}^{2}}{||u||_{6}^{2}}$ is attained by $U_{\varepsilon}(x)\!=\!3^{\frac{1}{4}}\left(\frac{\varepsilon}
{\varepsilon^{2}+|x|^{2}}\right)^{\frac{1}{2}}$ for $\varepsilon>0$, we have
\begin{align*}
&\frac{b{\mathcal{S}}^3}{2}
+\sqrt{a\mathcal{S}^3+\frac{b^2{\mathcal{S}}^6}{4}}=\frac{b{||\nabla U_{\varepsilon}||}_{2}^{6}}{2||U_{\varepsilon}||
_6^6}+\sqrt{\frac{a{||\nabla U_{\varepsilon}||}_{2}^{6}}{||U_{\varepsilon}||
_6^6}+\frac{b^2{||\nabla U_{\varepsilon}||}_{2}^{12}}{4||U_{\varepsilon}||
_6^{12}}}\\
&=e^{2t_{U_{\varepsilon},0}}{||\nabla U_{\varepsilon}||}_{2}^{2}\!\geq\!\inf_{u \in {{D}}^{1,2}(\mathbb{R}^{3})\setminus \{0\}}e^{2t_{u,0}}{||\nabla u||}_{2}^{2}.
\end{align*}
Then \eqref{eStIofE} is true. Similarly, we can prove $\inf _{u \in S_c}e^{2t_{u,0}}{|| u||}_{6}^{2}\!=\!\Lambda$. These facts imply that
\begin{align*} 
\inf _{u \in S_c}\Psi_{u}^{0}(t_{u,0})=\inf _{u \in S_c} \Big\{\frac{a}{2}e^{2t_{u,0}} ||\nabla u ||_{2}^2+\frac{b}{4}e^{4t_{u,0}} ||\nabla u ||_{2}^4-\frac{e^{ 6t_{u,0}}}{6} {||u||}_6^6\Big\}=\frac{a\mathcal{S}\Lambda}{3}\!+\!\frac{b\mathcal{S}^2{\Lambda}^2}{12}.
\end{align*}
Finally, we show that $\inf _{\mathcal{P}_{c,0}}E_0\!=\!\inf _{ \mathcal{P}_{c,0} \cap S_{c,r} } E_{0}$. Otherwise, there exists $u\!\in\! \mathcal{P}_{c,0} \!\setminus \!S_{c,r}$ with $E_{0}(u)\!<\!\inf _{ \mathcal{P}_{c,0} \cap S_{c,r} } E_{0}$. Then we let $v \!:=\!{|u|}^*$, the symmetric decreasing rearrangement of $|u|$, which lies in $S_{c,r}$.
Then, we have $E_{0}\left(v\right)\!\leq\! E_{0}\left(u\right)$ and $P_{0}\left(v\right)\!\leq\! P_{0}\left(u\right)\!=\!0$. If $P_{0}(v)\!=\!0$, then $E_{0}(u)\!<\!\inf _{ \mathcal{P}_{c,0} \cap S_{c,r} } E_{0}\!\leq\!E_{0}(v)$, a contradiction, and hence we get $P_{0}(v)\!<\!0$. By Lemma \ref{LeM3.4}, we have $t_v\!<\!0$. However, we get a contradiction that
\begin{align*}
E_{0}\left(u\right)<\inf _{ \mathcal{P}_{c,0} \cap S_{c,r} } E_{0}\leq E_{0}\left(t_v \star v\right)=\frac{a}{4}e^{2t_{v}} ||\nabla v ||_{2}^2+\frac{1}{12}e^{6t_v}
||v||_6^6 \leq\frac{a}{4}  ||\nabla u ||_{2}^2+\frac{1}{12}
||u||_6^6=E_{0}\left(u\right),
\end{align*}
where we used the fact that $t_v \star v$ and $u$ lies in $\mathcal{P}_{c,0}$. This proves that $m_r(c,0)=m(c,0)$.
\end{proof}

Based on Lemmas \ref{LeM3.4}-\ref{lEM7.4} and Proposition \ref{prp4.1}, we can prove Theorem \ref{tTH1.3}.

\noindent \textbf{Proof of Theorem \ref{tTH1.3}:}
The proof is different from that of Theorem \ref{th1.1}-(2), we should revise  the minimax class as
$$
\Gamma :=\left\{\gamma(\tau)=\big(\zeta(\tau),\beta(\tau)\big) \in C\left([0,1], \mathbb{R}\times S_{c,r}\right) ; \gamma(0) \in (0,\bar{A_k}), \gamma(1) \in (0, E_{\mu}^{0})\right\}.
$$
Then, it is standard as the proof of Theorem 1.6 in \cite{NsAe} that $E_{\mu}|_{S_c}$ has a critical point $\hat{u}_{c,\mu}$ at Mountain Pass level $\sigma(c, \mu)\!>\!0$ and $\hat{u}_{c,\mu}$ solves $(1.1)_{\hat{\lambda}_{c,\mu}}$ for some $\hat{\lambda}_{c,\mu}\!<\!0$. Similar to Lemma \ref{lemma6.5.1}, we get $\inf _{\mathcal{P}_{c,\mu}}E_\mu\!=\!\inf _{ \mathcal{P}_{c,\mu} \cap S_{c,r} } E_{\mu}$, then $\hat{u}_{c,\mu}$ is a ground state of $E_{\mu}|_{S_c}$. The proof of the asymptotic result is similar to that of Theorem \ref{th1.1}-(5). \qed

Theorem \ref{tTH1.4} is concerned with the Sobolev critical case $p\!=\!6$.  Proposition \ref{prp4.2} and Lemma \ref{lemma7.5} are crucial in the analysis. We first prove the existence results. \\
\noindent \textbf{Proof of Theorem \ref{tTH1.4}-(1),(2):}
Lemma \ref{lemma7.5} gives $m_{r}(c, \mu)\!<\!\frac{a\mathcal{S}\Lambda}{3}\!+\!\frac{b\mathcal{S}^2{\Lambda}^2}{12}$, the rest of the proof is the same as that of Theorem \ref{tTH1.3}, but we shall replace Proposition \ref{prp4.1} by Proposition \ref{prp4.2}.  \qed

\noindent \textbf{Proof of Theorem \ref{tTH1.4}-(3):}
Let us consider $\left\{\hat{u}_{\mu} : 0<\mu<\overline{\mu}\right\}$, with $\overline{\mu}$ small enough. From Theorem \ref{tTH1.4}-(1)(2) and Lemma  \ref{lemma6.5.1}, we know that
\begin{equation} \label{eqq6.11}
\begin{aligned}
\frac{a\mathcal{S}\Lambda}{3}\!+\!\frac{b\mathcal{S}^2{\Lambda}^2}{12}>E_{\mu}\left({\hat{u}_{\mu}} \right)=\frac{a}{4}||\nabla {\hat{u}_{\mu}}||_{2}^2+\mu(\frac{\delta_q}{4}-\frac{1}{q})
||{\hat{u}_{\mu}}||_q^q+\frac{1}{12}||{\hat{u}_{\mu}}||_6^6,
\end{aligned}
\end{equation}
This leads to $||\nabla\hat{u}_{\mu}||_2^{2}\leq C$. So $\left\{\hat{u}_{\mu}\right\}$ is bounded in $H^1$. Since each $\hat{u}_{\mu}$ is a positive radial function in $S_{c}$, we deduce that up to a subsequence $\hat{u}_{\mu} \rightharpoonup \hat{u}$ weakly in $H^1$,  strongly in $L^{r}$ for $2<r<6$ and a.e. on $\mathbb{R}^{3}$, as $\mu \rightarrow 0^{+}$. Using the fact that $\hat{u}_{\mu}$ solves
\begin{equation}\label{eqq6.12}
- \Bigl(a+b  {{{\| {\nabla \hat{u}_{\mu} } \|}_2^2}} \Bigr)\Delta \hat{u}_{\mu}
   =\hat{\lambda}_{\mu} \hat{u}_{\mu}+ {| \hat{u}_{\mu} |^{4}}\hat{u}_{\mu}+\mu {| \hat{u}_{\mu} |^{q - 2}}\hat{u}_{\mu} \text { in } \mathbb{R}^{3}
\end{equation}
for $\hat{\lambda}_{\mu}<0$ and $P_{\mu}\left(\hat{u}_{\mu}\right)=0$, we infer that
$$\hat{\lambda}_{\mu} c^{2}\!=\!a{\| {\nabla \hat{u}_{\mu}} \|}_2^2+b{\| {\nabla \hat{u}_{\mu}} \|}_2^4-
\mu{||\hat{u}_{\mu}||}_q^{q}-{||\hat{u}_{\mu}||}_6^{6}\!=\!\mu(\delta_q-1)||\hat{u}_{\mu}||_{q}^{q} \to 0~~~~~~~~\mbox{as}~~~~\mu\to 0^{+}.$$
Therefore, we have $\mathop {\lim }\limits_{\mu \rightarrow 0^{+}}\{a{\| {\nabla \hat{u}_{\mu}} \|}_2^2+b{\| {\nabla \hat{u}_{\mu}} \|}_2^4\}=\mathop {\lim }\limits_{\mu \rightarrow 0^{+}}{||\hat{u}_{\mu}||}_6^{6}=\ell\geq0$
and $\hat{\lambda}_{\mu}\to0$. So $\mathop {\lim }\limits_{n  \to \infty} {\| {\nabla \hat{u}_{\mu}} \|}_2^2=\sqrt{\frac{\ell}{b}+\frac{a^2}{4b^2}}-\frac{a}{2b}$ and by the Sobolev inequality $\ell\geq b{\mathcal{S}}^2 \ell^{\frac{2}{3}}+a\mathcal{S}\ell^{\frac{1}{3}}$.

If $\ell=0$, then we have
$\hat{u}_{\mu} \rightarrow 0$ strongly in $D^{1,2}({\R}^3)$ and so
$E_{\mu}\left(\hat{u}_{\mu}\right)\to0$ as $\mu\to 0^{+}$. Imitate Lemma \ref{lem6.8}, we can prove that $\sigma(c, \mu)$  is monotone decreasing in $\mu$ and
$$\lim _{\mu \rightarrow 0^{+}} E_{\mu}\left(\hat{u}_{\mu}\right)=\lim _{\mu \rightarrow 0^{+}} \sigma(c, \mu) \geq \sigma(c, \overline{\mu})>0,$$
the contradiction implies that $\ell\not=0$ and so we have $\ell\geq\Lambda^3$. By using the monotonicity of $\sigma(c, \mu)$ and (\ref{algn6.4.1}), we also have
\begin{align*}
\frac{a\mathcal{S}\Lambda}{3}\!+\!\frac{b\mathcal{S}^2{\Lambda}^2}{12}&\leq \frac{\ell}{12}+\frac{a}{4}
\Big(\sqrt{\frac{\ell}{b}+\frac{a^2}{4b^2}}-\frac{a}{2b}\Big)=\lim _{\mu \rightarrow 0^{+}}\left[\frac{a}{4}||\nabla {\hat{u}_{\mu}}||_{2}^2+\frac{1}{12}||{\hat{u}_{\mu}}||_6^6+
\mu(\frac{\delta_q}{4}-\frac{1}{q})
||{\hat{u}_{\mu}}||_q^q\right] \\
&=\lim _{\mu \rightarrow 0^{+}} E_{\mu}\left(\hat{u}_{\mu}\right)=\lim _{\mu \rightarrow 0^{+}} \sigma(c, \mu) \leq m_r(c,0)=\frac{a\mathcal{S}\Lambda}{3}\!+\!\frac{b\mathcal{S}^2{\Lambda}^2}{12},
\end{align*}
which implies that $\ell=\Lambda^3$, $||{\hat{u}_{\mu}}||_6^6 \to {\Lambda}^3$ and ${\| {\nabla \hat{u}_{\mu}} \|}_2^2\to \mathcal{S}{\Lambda}$ as $\mu\to 0^{+}$.  \qed

\end{document}